\documentclass[final,leqno]{siamltex}
\usepackage{amsmath}
\usepackage{graphicx}
\usepackage{mathrsfs}
\usepackage{bm}
\usepackage{float}
\usepackage{amsfonts,amssymb}
\usepackage{stmaryrd} 
\usepackage{dsfont}
\usepackage{pifont}
\usepackage{tikz}
\usepackage{wrapfig} 
\usepackage{hyperref}
\usepackage{multirow}
\usepackage{lineno}
\usepackage{mathtools}
\usepackage{appendix}
\usepackage{color}
\usepackage{subfigure}

\newtheorem{remark}{Remark}

\def\3bar{{|\hspace{-.02in}|\hspace{-.02in}|}}

\newcommand{\vertiii}[1]{{\left\vert\kern-0.25ex\left\vert\kern-0.25ex\left\vert #1
    \right\vert\kern-0.25ex\right\vert\kern-0.25ex\right\vert}}
\def\T{\mathcal{T}}

\title{Generalized Weak Galerkin methods for Stokes equations}

\begin{document}
\author{
Wenya Qi
\thanks{School of Mathematics and Information Sciences, Henan Normal University, Xinxiang, Henan 453007, China(qiwymath@163.com).
The research of Qi was partially supported by Natural Science Foundation of Henan Province(Grant No.222300420213).}
 \and Padmanabhan Seshaiyer
  \thanks{Department of Mathematical Sciences, George Mason University, Fairfax, VA 22030, USA(pseshaiy@gmu.edu).}
 \and Junping Wang
\thanks{Division of Mathematical Sciences, National Science Foundation, Alexandria, VA 22314, USA(jwang@nsf.gov). The research of Wang was supported by the NSF IR/D program, while working at National Science Foundation. However, any opinion, finding, and conclusions or recommendations expressed in this material are those of the author and do not necessarily reflect the views of the National Science Foundation.}}

\maketitle
\baselineskip=12pt

\begin{abstract}
A new weak Galerkin finite element method, called generalized weak Galerkin method ({g}WG), is introduced for Stokes equations in this paper by using a new definition of the weak gradient. Error estimates in energy norm and $L^2$ norm for the velocity and $L^2$ norm for the pressure are derived for elements with arbitrary combination of polynomials. Some numerical examples are presented to verify the effectiveness, theoretical convergence orders, and robustness of the proposed scheme.
\end{abstract}



\pagestyle{myheadings}

\section{Introduction}\label{sec:1}
Most physical phenomena governing fluid flow motion in various real-world applications arising in biology, medicine, meteorology, oceanography, manufacturing, etc. can be described using the Navier-Stokes equations \cite{Temam2001}. Some of these applications include, for example, viscous flow past a cubic array of spheres \cite{HASIMOTO1958},  the  dynamics of airflow over an airplane wing in engineering \cite{Nordanger2015},
the flow of air currents in meteorology \cite{Poul2008}, fluid flow in ocean circulation \cite{Hill1996}, the transport of virus in the air or other medium \cite{Li2020} and other multiphysics applications \cite{Aulisa2008}.
There are many other physical applications where the advective inertial forces are small in comparison to viscous forces which can be modeled by
Stokes equations instead of Navier Stokes equations. Such applications include swimming of microorganisms to the flow of lava.

One of the most powerful numerical methods that has been well studied for numerically solving the Stokes equations is finite element method \cite{Arnold1984, Girault1979, Verfurth1989, Stenberg1990, Kechkar1992}.
A $hp$ discontinuous Galerkin method was analyzed in \cite{Toselli2002} for Stokes equations with a positive-definite non-symmetric velocity bilinear form
on quadrilateral or hexahedral mesh.
In \cite{Cockburn2002}, a local discontinuous Galerkin approximation was employed for Stokes equations.
By modifying saddle-point Lagrangian functional, stabilized mixed finite element methods were presented in \cite{Bochev2006}
and the optimal error analysis was derived on lowest order elements for Stokes equations.
A new finite element method with H(div) elements was developed and analyzed for Stokes equations in \cite{wang2007new, wang2009a}.
The optimal orders were obtained for the velocity and pressure approximation and some examples of H(div) elements were provided.
In \cite{Kondratyuk2008}, an adaptive finite element method with three nested loops was introduced and analyzed for Stokes equations.
In view of polynomials with the same degree for the velocity and pressure, a priori error estimate of hybridizable discontinuous Galerkin method
 was established for Stokes equations in \cite{Cockburn2011}.
In \cite{wang2013stokes}, a general finite volume formulation has been derived for Stokes equations with a unified posteriori error estimator.
By employing a weakly over-penalized symmetric interior penalty for velocity, a mixed finite element method was introduced for Stokes problem in \cite{Barker2014},
and the error estimates were established based on the discontinuous linear element for velocity and constant for pressure dealing with more general nonconforming meshes.
A coupled mortar finite element formulation was presented for incompressible Stokes equations in primal velocity-pressure variables where the local approximations were designed using divergence stable hp-mixed finite elements \cite{Chilton2002}.
By constructing a conforming and divergence-free finite element, an error estimate for Stokes equations was derived with the lowest order element in \cite{Guzman2014}.
A virtual element space was considered for Stokes equations in \cite{Veiga2017}.

Weak Galerkin (WG) finite element method has been developed for second order elliptic equations in \cite{Wang2013weak, Wang2014a} by introducing weak functions and discrete weak gradients.
The weak Galerkin method uses the discontinuous functions and employs the polygonal or polyhedron mesh partition which makes the formulation as simple as conforming finite element method.
One  advantage of the weak Galerkin method is that the polynomial approximation and the shape of element can be chosen flexibly, and another advantage is that the numerical formulation is simpler to implement.
By using the element $([P_{k+1}(T)]^d,[P_{k}(\partial T)]^d)$ for the velocity and $P_{k}(T)$ for the pressure,
a weak Galerkin method has been established for the Stokes equations in \cite{wang2016stokes}, where the weak gradient space was polynomial $[P_{k}(T)]^{d\times d}$
  and weak divergence space was $ P_{k}(T)$ for velocity. Here, $k\geq 0$ is arbitrary non-negative integer.
  Moreover, the partition of domain was composed of arbitrary polygons/polyhedra satisfying some shape regular conditions.
  In \cite{Mu2018},  for the lowest order weak Galerkin element, a divergence free formulation was constructed to eliminate pressure variable.
  A simplified weak Galerkin method was presented for Stokes equations in \cite{liu2019} and the optimal error estimates were established in $H^1$ and $L^2$ norm.
 In \cite{Bao2019},  a posteriori error estimator applied on general polygonal meshes was discussed for Stokes equations where the elements was $([P_{k+1}(T)]^d,[P_{k+1}(\partial T)]^d)$ for the velocity
 and $P_{k}(T)$ for the pressure.
 Using a velocity reconstruction process,  a pressure-robust weak Galerkin finite element method was developed in \cite{Mu2020} and optimal order of convergence was established. 

 In this paper, we will extent the weak Galerkin finite element spaces to arbitrary polynomial combination by using a new weak gradient definition for Stokes equations and analyze the convergence.
 For the classical weak Galerkin method, the polynomials of elements have been flexible with general polygon mesh.
 At first,  the elements $(P_{k}(T),P_{k+1}(\partial T),[P_{k+1}(T)]^d)$ and $(P_{k}(T),P_{k}(\partial T),RT_{k}(T))$  were employed for second order elliptic problems in \cite{Wang2013weak}.
The weak Galerkin method has employed the element $(P_{k+1}(T),P_{k+1}(\partial T),[P_{k}(T)]^d)$ for elliptic interface problems in \cite{Mu2016A}  with a parameter free stabilizer.
 Then,  by using a stabilizer,  a polynomial reduction weak Galerkin element $(P_{k+1}(T),P_{k}(\partial T),[P_{k}(T)]^d)$ was presented in \cite{Mu2015}.
 In \cite{wang2018a}, the arbitrary polynomials elements were considered with a new stabilizer for elliptic problems, and the classical weak gradient definition was introduced.
Despite all these efforts, there are still some restrictions on the polynomial orders of weak functions and weak gradients. The generalized weak Galerkin method in this paper is based on a new definition of the weak gradient and shall relax all these existing restrictions. For example, the element $(P_1(T),P_0(\partial T),[P_1(T)]^d)$ has no convergence with the classical weak gradient in \cite{wang2018a}, but it does converge with the generalized weak Galerkin method.
 For Stokes equations we will introduce two stabilizers; one for the velocity and the other for the pressure. Optimal order of convergence is establised for the velocity approximation in $H^1$ and $L^2$ norms, as well as for the pressure approximation in $L^2$.

The paper is organized as follows. In Section \ref{sec:2}, a generalized weak gradient is defined and some preliminary properties are presented.
In Section \ref{sec:3}, the generalized weak Galerkin scheme is formulated for Stokes equations and a new {\em inf-sup} condition is derived to deal with the uniqueness and existence of the numerical solutions.
In Section \ref{sec:4}, some error equations  are  derived by using the properties of various $L^2$ projection operators.
In Section \ref{sec:5}, some error estimates are first established for the velocity in $H^1$ norm and for the pressure in $L^2$ norm, and then in $L^2$ norm for the velocity approximation through a duality argument.
In Section \ref{sec:6}, some numerical results are presented to illustrate the accuracy and efficiency of  our method.

In this paper, we denote positive numbers by $\epsilon_i > 0$,  and $C$ is reserved for a generic constant independent of the mesh size $h$. We also follow the usual notations for Sobolev spaces and the corresponding norms \cite{adams2003sobolev}.

\section{ Generalized {\color{blue}{Weak}} Gradient and Divergence}\label{sec:2}
Consider the model Stokes problem that seeks the velocity $\mathbf{u}$ and the pressure $p$ satisfying
\begin{equation}\label{stokeproblem}
\begin{aligned}
-\nabla \cdot(A \nabla \mathbf{u})+\nabla p&=\mathbf{f},~~~~~~~~~~\mbox{in}~\Omega, \\
\nabla\cdot\mathbf{u}&=0,~~~~~~~~~~\mbox{in}~\Omega,\\
\mathbf{u}&=0, ~~~~~~~~~~\mbox{on}~\partial\Omega,
\end{aligned}
\end{equation}
where $\Omega$ is an open bounded polygonal or polyhedral domain in $\mathbb{R}^d\ (d=2,3)$, $\mathbf{f}$ is the external body force, and
$A$ is a $d\times d$ matrix-valued function with respect to the fluid viscosity in $\Omega$. Assume that $A$ is symmetric and uniformly positive definite; i.e., there exist two positive numbers
$\lambda_1, ~\lambda_2$ such that
\begin{equation*}
\begin{aligned}
\lambda_1\xi^t\xi \leq \xi^tA\xi \leq \lambda_2\xi^t\xi,~~\xi\in \mathbb{R}^d,
\end{aligned}
\end{equation*}
where $\xi^t$ is the transpose of the column vector $\xi$.

A weak formulation of the problem \eqref{stokeproblem} reads as follows: find $\mathbf{u}\in[H^1_0(\Omega)]^d$ and $p\in L_0^2(\Omega)$ satisfying
\begin{equation}\label{weakform}
\begin{aligned}
(A \nabla \mathbf{u}, \nabla \mathbf{v})-(p,\nabla\cdot\mathbf{v})&=(\mathbf{f},\mathbf{v}), ~~~~~~\forall \mathbf{v}\in[H^1_{0}(\Omega)]^d,\\
(\nabla\cdot\mathbf{u}, q)&=0,~~~~~~~~~~~\forall q\in L_0^2(\Omega).
\end{aligned}
\end{equation}
Here, we denote the Sobolev spaces by
\begin{equation*}
\begin{aligned}
[H^1_{0}(\Omega)]^d&=\{\mathbf{v}\in [H^1(\Omega)]^d, ~\mathbf{v}=\mathbf{0} ~\mbox{on} ~\partial \Omega\},\\
L_0^2(\Omega)&=\Big\{q\in L^2(\Omega), \int_{\Omega}q dx = 0\Big\}.
\end{aligned}
\end{equation*}

Let $\mathcal{T}_h$ be a polygonal or polyhedral partition of the domain $\Omega$ satisfying the shape regular conditions in \cite{Wang2014a}. For each element $T\in\mathcal{T}_h$, $h_T$ is its diameter and $h=\displaystyle\max_{T\in\mathcal{T}_h}h_T$ is the mesh size of the partition $\mathcal{T}_h$. Denote by $\partial T$ the edges or flat faces of the element $T\in\T_h$ and $\mathcal{E}_{I}$ the set of all interior edges or flat faces. For each edge or flat face $e\in \mathcal{E}_{I}$, denote by $h_e$ its diameter. Denote by ${P}_{k}(T)$ the polynomial space of degree less than or equal to $k$ in all variables on the element $T$.
We consider the following finite element spaces,
\begin{equation*}
\begin{aligned}
 \mathbf{V}_{h}:=&\{(\mathbf{v}_0, \mathbf{v}_b): \mathbf{v}_0|_{T}\in[{P}_{k}(T)]^d, \mathbf{v}_b|_{\partial T}\in[{P}_{j}(\partial T)]^d,~ \forall T \in \mathcal{T}_{h}\},\\
W_{h}:=&\{q\in L_0^2(\Omega): q|_{T}\in{P}_{n}(T), ~\forall T \in \mathcal{T}_{h}\},
 \end{aligned}
\end{equation*}
where $k, j, n\geq 0$ are arbitrary non-negative integers.
Denote by $\mathbf{V}_{h}^0$ the subspace of $\mathbf{V}_{h}$ consisting of functions with vanishing boundary value; i.e.,
$$
\mathbf{V}_{h}^0=\{\mathbf{v}\in \mathbf{V}_{h}, ~\mathbf{v}_b=\mathbf{0} ~\mbox{on} ~\partial \Omega\}.
$$

Denote by $Q_h =\{Q_0, Q_b\}$ the operator onto $\mathbf{V}_{h}$, where $Q_0: [L^2(T)]^d \rightarrow [P_k(T)]^d$ and $Q_b: [L^2(\partial T)]^d \rightarrow [P_j (\partial T)]^d$ are the
local $L^2$ projection operators. Furthermore, denote by $Q_h^p: L_0^2(\Omega) \rightarrow W_h$ the usual $L^2$ projection operator onto $W_h$.

\begin{definition}\label{weak-grad}
For each $\mathbf{v}\in \mathbf{V}_h$, the generalized discrete weak gradient, denoted by $\nabla_w\mathbf{v}$, is given on each element $T\in\T_h$  by
\begin{equation}\label{weakgradient}
\nabla_w\mathbf{v}|_T=\nabla\mathbf{v}_0|_T+\delta_w\mathbf{v}|_T,
\end{equation}
where $\delta_w\mathbf{v} |_T\in [P_l(T)]^{d\times d}$ is computed as follows:
\begin{equation}\label{weakgradient-2}
(\delta_w\mathbf{v}, \phi)_T=\langle \mathbf{v}_b- Q_b\mathbf{v}_0,  \phi\cdot\mathbf{n}  \rangle_{\partial T}, ~\forall \phi\in [P_l(T)]^{d\times d},
\end{equation}
where $\mathbf{n}$ denotes the unit outward normal vector to $\partial T$.
\end{definition}

\begin{definition}\label{weakdivergence}
For each $\mathbf{v}\in \mathbf{V}_h$, the discrete weak divergence $\nabla_w\cdot\mathbf{v}\in P_m(T)$ is defined such that
\begin{equation}\label{weak-div}
(\nabla_w\cdot\mathbf{v}, \psi)_T = - (\mathbf{v}_0, \nabla \psi)_T+\langle \mathbf{v}_b, \psi\cdot\mathbf{n} \rangle_{\partial T}, ~\forall\psi\in P_m(T).
\end{equation}
 \end{definition}

\section{Numerical Scheme of the Generalized Weak Galerkin}\label{sec:3}
The generalized weak Galerkin finite element method for the Stokes problem \eqref{stokeproblem} seeks $\mathbf{u}_h\in \mathbf{V}_{h}^0$ and $p_h\in W_{h}$ satisfying
\begin{equation}\label{sch:WG_FEM}
\left\{
\begin{aligned}
(A \nabla_w \mathbf{u}_h, \nabla_w \mathbf{v})-(p_h,\nabla_w\cdot\mathbf{v})+s_1(\mathbf{u}_h, \mathbf{v})&=(\mathbf{f},\mathbf{v}_0), ~~~~~~\forall~ \mathbf{v}\in\mathbf{V}_{h}^0,\\
(\nabla_w\cdot\mathbf{u}_h, q_h)+s_2(p_h, q_h)&=0,~~~~~~~~~~~~~\forall~q_h\in W_{h}.
\end{aligned}\right.
\end{equation}
Here, the stabilizing terms are defined as follows
\begin{equation*}
\begin{aligned}
s_1(\mathbf{u}_h, \mathbf{v})&=\sum_{T\in\mathcal{T}_h}h_T^{-{\gamma}}\langle \mathbf{u}_b-Q_b\mathbf{u}_0, \mathbf{v}_b-Q_b\mathbf{v}_0\rangle_{\partial T},\\
s_2(p_h, q_h)&=\mu\sum_{e\in\mathcal{E}_{I}}h_e^{-\beta}\langle \llbracket p_h \rrbracket, \llbracket q_h \rrbracket \rangle_e,
\end{aligned}
\end{equation*}
where $\llbracket q_h \rrbracket=q_h|_{\partial T_1\cap e}-q_h|_{\partial T_2\cap e}$ denotes the jump on the edge or flat face $e\in \mathcal{E}_{I}$ shared by two elements $T_1$ and $T_2$.
Assume that $\vert\gamma\vert<\infty$, $\vert\beta\vert<\infty$,  and $\mu$ is a non-negative number.

For each $\mathbf{v}\in \mathbf{V}_h$, we define the energy semi-norm by
$$\interleave \mathbf{v} \interleave^2=(A \nabla_w \mathbf{v}, \nabla_w \mathbf{v})+s_1(\mathbf{v}, \mathbf{v}).$$

\begin{lemma}\label{lem:norm}
$\interleave \cdot\interleave$ defines a norm in the space $\mathbf{V}_h^0$.
\end{lemma}
\begin{proof}
It suffices to show that for each $\mathbf{v}\in V^0_h$,  $\interleave \mathbf{v}\interleave=0$ implies $\mathbf{v}\equiv 0$. To this end, assume $\interleave \mathbf{v}\interleave=0$. From its definition we have $\nabla_w\mathbf{v}=0$ and $\mathbf{v}_b=Q_b\mathbf{v}_0$. For each $\phi\in[P_l(T)]^{d\times d}$,
we get from Definition \ref{weakgradient}
\begin{equation*}
\begin{aligned}
(\delta_w\mathbf{v},\phi)=\langle \mathbf{v}_b- Q_b\mathbf{v}_0,  \phi\cdot\mathbf{n}  \rangle=0,
\end{aligned}
\end{equation*}
It follows that $\delta_w\mathbf{v}=0$, which further leads to $0=\nabla_w\mathbf{v}=\nabla\mathbf{v}_0+\delta_w\mathbf{v}=\nabla\mathbf{v}_0$. Thus, $\mathbf{v}_0$ must be a piecewise constant vector field so that $\mathbf{v}_0=Q_b\mathbf{v}_0 = \mathbf{v}_b$ on each interior edge/face $e$. Finally, from $\mathbf{v}\in \mathbf{V}_h^0$ we have $\mathbf{v}\equiv 0$.
\end{proof}

In order to derive an {\em inf-sup} condition, we denote the bilinear form $b(\mathbf{v}, q_h):=(\nabla_w\cdot\mathbf{v}, q_h)$.
\begin{lemma}\label{lem:inf-sup}
Assume that $n\leq \min \{m, k+1\}$.
For each $q_h\in W_h$, there exists $\mathbf{v}\in \mathbf{V}_h^0$ such that
\begin{equation}\label{inf_supconditon}
\left\{
\begin{aligned}
b(\mathbf{v}, q_h) &\geq \frac{1}{2}\Vert q_h\Vert^2-Ch^{\beta+1}s_2(q_h,q_h),\\
\interleave \mathbf{v}\interleave&\leq C(1+h^{\frac{1-\gamma}{2}}) \Vert q_h\Vert.
\end{aligned}\right.
\end{equation}
\end{lemma}
\begin{proof}
For any given $q_h\in W_h$, there exists $\mathbf{w}\in [H_0^1(\Omega)]^d$ such that $\nabla\cdot\mathbf{w}=q_h$ and $\Vert \nabla \mathbf{w} \Vert \leq C\Vert q_h\Vert$. By choosing $\mathbf{v}=Q_h\mathbf{w}$ and using the definition of the weak divergence  \eqref{weak-div}, the property of $L^2$ projections, the Cauchy-Schwarz inequality, and  the trace inequality, we have from $n\leq  \min \{m, k+1\}$,
\begin{equation*}
\begin{aligned}
b(Q_h\mathbf{w}, q_h)&=\sum_{T\in\mathcal{T}_h}\Big((-Q_0\mathbf{w}, \nabla q_h)_T+\langle Q_b \mathbf{w}\cdot \mathbf{n}, q_h \rangle_{\partial T}\Big)\\
&=\sum_{T\in\mathcal{T}_h}\Big((-\mathbf{w}, \nabla q_h)_T+\langle Q_b \mathbf{w}\cdot \mathbf{n}, q_h \rangle_{\partial T}\Big)\\
&=\sum_{T\in\mathcal{T}_h}\Big((\nabla\cdot\mathbf{w}, q_h)_T-\langle (I-Q_b) \mathbf{w}\cdot \mathbf{n}, q_h \rangle_{\partial T}\Big)\\
&=\sum_{T\in\mathcal{T}_h}\Vert q_h\Vert^2_T-\sum_{e\in\mathcal{E}_{I}}\langle (I-Q_b) \mathbf{w}\cdot \mathbf{n}_e, \llbracket q_h\rrbracket  \rangle_{e}\\
&\geq \Vert q_h\Vert^2-\sum_{e\in\mathcal{E}_{I}}(\epsilon_1h_e^{-1}\Vert(I-Q_b)\mathbf{w}\Vert_e^2+\frac{1}{4\epsilon_1}h_e\Vert\llbracket q_h\rrbracket \Vert ^2_{e})\\
&\geq \Vert q_h\Vert^2-\epsilon_1C\Vert\nabla\mathbf{w}\Vert^2-\sum_{e\in\mathcal{E}_{I}}\frac{1}{4\epsilon_1}h_e\Vert\llbracket q_h\rrbracket \Vert ^2_{e},
\end{aligned}
\end{equation*}
where the constant $\epsilon_1> 0$ and $\mathbf{n}_e$ is the unit fixed direction normal to the edge $e\in\mathcal{E}_{I}$.
By taking $\epsilon_1=\frac{1}{2C}$, we arrive at the following inequality
\begin{equation}\label{inf_sup1}
\begin{aligned}
b(Q_h\mathbf{w}, q_h)&\geq \frac{1}{2}\Vert q_h\Vert^2-C\sum_{e\in\mathcal{E}_{I}}h_e\Vert\llbracket q_h\rrbracket \Vert ^2_{e},
\end{aligned}
\end{equation}
which asserts the first inequality in \eqref{inf_supconditon}.

To verify the second inequality in \eqref{inf_supconditon}, we have from the definition of the weak gradient {\color{blue}\eqref{weakgradient}} that
$$\Vert \nabla_w Q_h\mathbf{w} \Vert=\Vert \nabla Q_0\mathbf{w} +\delta_w Q_h\mathbf{w}\Vert\leq\Vert \nabla Q_0\mathbf{w} \Vert+\Vert\delta_w Q_h\mathbf{w}\Vert.$$
It is not hard to see that $\Vert \nabla Q_0\mathbf{w} \Vert\leq C \Vert \nabla \mathbf{w} \Vert$.  To analyze the second term in the above inequality, we use Definition {\color{blue}\ref{weak-grad}} and the trace inequality to obtain
\begin{equation*}
\begin{aligned}
\Vert\delta_w Q_h\mathbf{w}\Vert_T^2&=\langle Q_b\mathbf{w}-Q_bQ_0\mathbf{w}, \delta_wQ_h\mathbf{w}\cdot\mathbf{n}\rangle_{\partial T}\\
&\leq\Vert Q_b(I-Q_0)\mathbf{w}\Vert_{\partial T}\Vert  \delta_w Q_h\mathbf{w}\Vert_{\partial T}\\
&\leq C \Vert \nabla \mathbf{w}\Vert _T\Vert\delta_w Q_h\mathbf{w}\Vert_T,
\end{aligned}
\end{equation*}
which implies $\Vert\delta_w Q_h\mathbf{w}\Vert\leq C\Vert \nabla \mathbf{w}\Vert$, and further leads to $\Vert \nabla_w Q_h\mathbf{w} \Vert\leq C\Vert \nabla \mathbf{w}\Vert$.

Next, by using the trace inequality and the property of $L^2$ projections we obtain
\begin{equation*}
\begin{aligned}
s_1(Q_h\mathbf{w}, Q_h\mathbf{w})&=\sum_{T\in\mathcal{T}_h}h_T^{-\gamma}\langle Q_b(I-Q_0)\mathbf{w}, (I-Q_0)\mathbf{w}\rangle_{\partial T}\\
&\leq \sum_{T\in\mathcal{T}_h}h_T^{-\gamma}\Vert Q_b(I-Q_0)\mathbf{w} \Vert_{\partial T} \Vert (I-Q_0)\mathbf{w} \Vert_{\partial T}\\
&\leq C\sum_{T\in\mathcal{T}_h}h_T^{-\frac{\gamma}{2}}\Vert Q_b(I-Q_0)\mathbf{w} \Vert_{\partial T}h_T^{\frac{1-\gamma}{2}} \Vert\nabla\mathbf{w}\Vert_T.
\end{aligned}
\end{equation*}
Thus, we have $s_1(Q_h\mathbf{w}, Q_h\mathbf{w})^{\frac{1}{2}}\leq Ch^{\frac{1-\gamma}{2}}\Vert\nabla\mathbf{w}\Vert$.  Finally, combining the above estimates with
$$
\interleave Q_h\mathbf{w}\interleave^2=(A \nabla_w Q_h\mathbf{w}, \nabla_w Q_h\mathbf{w})+s_1(Q_h\mathbf{w}, Q_h\mathbf{w}),
$$
gives $\interleave Q_h\mathbf{w}\interleave\leq C(1+h^{\frac{1-\gamma}{2}})\Vert\nabla\mathbf{w}\Vert\leq C(1+h^{\frac{1-\gamma}{2}})\Vert q_h\Vert$. This completes the proof of the lemma.
\end{proof}
\medskip

The following is our main result on the solvability of the generalized weak Galerkin finite element method.

\begin{theorem}\label{th:nique}
The generalized weak Galerkin method \eqref{sch:WG_FEM} has one and only one solution.
\end{theorem}
\begin{proof}
As the number of equations is the same as the number of unknowns in \eqref{sch:WG_FEM}, it suffices to establish the solution uniqueness. To this end, assume $\mathbf{u}_h\in \mathbf{V}_{h}$ and $p_h\in W_{h}$ satisfies \eqref{sch:WG_FEM} with homogeneous data (i.e., $\mathbf{f}=0$ and $\mathbf{g}=0$). By choosing $\mathbf{v} = \mathbf{u}_h$ and $q_h = p_h$ in \eqref{sch:WG_FEM}, we arrive at
$$
(A \nabla_w \mathbf{u}_h, \nabla_w \mathbf{u}_h)+s_1(\mathbf{u}_h, \mathbf{u}_h)+s_2(p_h, p_h)=0.
$$
It follows from Lemma \ref{lem:norm} that $\mathbf{u}_h\equiv 0$. Furthermore, from Lemma \ref{lem:inf-sup} and the fact that $s_2(p_h, p_h)=0$ we arrive at $\|p_h\|=0$ so that $p_h\equiv 0$. This completes the proof of the solution uniqueness.
\end{proof}

\section{Error Equations}\label{sec:4}

Denote by $\bm{Q}_s: [L^2(T)]^{d\times d} \rightarrow [P_s(T)]^{d\times d}$ the $L^2$ projection operator.

 \begin{lemma}\label{lem:identity}
For $\mathbf{v}\in \mathbf{V}_h$ and $\mathbf{w}\in [H^1(T)]^d$ on each element $T\in\mathcal{T}_h$, the following identities hold true
 \begin{equation*}
\begin{aligned}
(\nabla_w\mathbf{v}, \phi)_T&=-(\mathbf{v}_0, \nabla\cdot\phi)_T+\langle \mathbf{v}_b, \phi\cdot\mathbf{n}\rangle_{\partial T},\\
(\nabla_wQ_h\mathbf{w}, \phi)_T&=(\nabla\mathbf{w}, \phi)_T+((I-Q_0)\mathbf{w}, \nabla\cdot\phi)_T,
 \end{aligned}
\end{equation*}
for any $\phi \in [P_s(T)]^{d\times d}$ with $s=\min\{j, l\}$.
 \end{lemma}
 \begin{proof}
  For $\mathbf{v}\in \mathbf{V}_h$, from the generalized weak gradient {\color{blue}\eqref{weakgradient}} we have
 \begin{equation*}
\begin{aligned}
(\nabla_w\mathbf{v}, \phi)_T&=(\nabla \mathbf{v}_0, \phi)_T+\langle \mathbf{v}_b-Q_b\mathbf{v}_0, \phi\cdot\mathbf{n}\rangle_{\partial T}\\
&=(\nabla \mathbf{v}_0, \phi)_T-\langle \mathbf{v}_0, \phi\cdot\mathbf{n}\rangle_{\partial T}+\langle \mathbf{v}_b, \phi\cdot\mathbf{n}\rangle_{\partial T}\\
&=-(\mathbf{v}_0, \nabla\cdot\phi)_T+\langle \mathbf{v}_b, \phi\cdot\mathbf{n}\rangle_{\partial T}, ~\forall \phi \in [P_s(T)]^{d\times d}.
 \end{aligned}
\end{equation*}
Next, using Definition {\color{blue}\ref{weak-grad}} again, we get the following equation from the property of $L^2$ projection
 \begin{equation*}
\begin{aligned}
(\nabla_wQ_h\mathbf{w}, \phi)_T&=(\nabla Q_0\mathbf{w}, \phi)_T+\langle Q_b\mathbf{w}-Q_bQ_0\mathbf{w}, \phi\cdot\mathbf{n}\rangle_{\partial T}\\
&=-(Q_0\mathbf{w}, \nabla\cdot\phi)_T+\langle Q_0\mathbf{w}, \phi\cdot\mathbf{n}\rangle_{\partial T}+\langle (I-Q_0)\mathbf{w}, \phi\cdot\mathbf{n}\rangle_{\partial T}\\
&=-(Q_0\mathbf{w}, \nabla\cdot\phi)_T+\langle \mathbf{w}, \phi\cdot\mathbf{n}\rangle_{\partial T}\\
&=-(Q_0\mathbf{w}, \nabla\cdot\phi)_T+(\nabla\mathbf{w}, \phi)_{T}+(\mathbf{w}, \nabla\cdot\phi)_{T}\\
&=(\nabla\mathbf{w}, \phi)_T+((I-Q_0)\mathbf{w}, \nabla\cdot\phi)_T,
 \end{aligned}
\end{equation*}
for all $\phi \in [P_s(T)]^{d\times d}$.
 \end{proof}

Let $(\mathbf{u}, p)$ be the solution of the Stokes problem \eqref{stokeproblem} and $(\mathbf{u}_h, p_h)$ be its numerical approximation arising from the generalized weak Galerkin scheme $\eqref{sch:WG_FEM}$.
Denote by $e_h^{\mathbf{u}}=Q_h\mathbf{u}-\mathbf{u}_h$ and $e_h^p=Q_h^p p-p_h$ the error functions.
Introduce the following linear functionals:
\begin{equation*}
\begin{aligned}
l_1(\mathbf{u}, \mathbf{v})&=\sum_{T\in\mathcal{T}_h}\langle (I-\bm{Q}_s)A\nabla\mathbf{u}\cdot\mathbf{n}, \mathbf{v}_0-\mathbf{v}_b\rangle_{\partial T},\\
l_2(p, \mathbf{v})&=\sum_{T\in\mathcal{T}_h}\langle (I-Q_h^p)p, (\mathbf{v}_b-\mathbf{v}_0)\cdot\mathbf{n}\rangle_{\partial T},\\
l_3(\mathbf{u}, \mathbf{v})&=((\bm{Q}_s-I)A\nabla\mathbf{u}, \nabla\mathbf{v}_0)+((I-Q_0)\mathbf{u}, \nabla\cdot \bm{Q}_sA\nabla_w\mathbf{v}),\\
l_4(\mathbf{u}, \mathbf{v})&=(A\nabla_wQ_h\mathbf{u}, (I-\bm{Q}_s)\nabla_w\mathbf{v}),\\
l_5(\mathbf{u}, q_h)&=\sum_{T\in\mathcal{T}_h}\langle (Q_b-I)\mathbf{u}\cdot\mathbf{n}, q_h\rangle_{\partial T}.
\end{aligned}
\end{equation*}
Observe that when $j\ge n$, we have $l_5(\mathbf{u}, q_h)=0$.

\begin{lemma}\label{lem:errorequ}
Let $(\mathbf{u}, p)$ be the exact solution of the Stokes problem \eqref{stokeproblem} and $(\mathbf{u}_h, p_h)$ the numerical solution
of the generalized weak Galerkin method $\eqref{sch:WG_FEM}$.
 Assume that $k-1\leq n\leq\min\{m,k+1\}$.
Then the following error equations hold true
\begin{equation*}
\begin{aligned}
(A\nabla_w e_h^{\mathbf{u}}, \nabla_w \mathbf{v})-(e^p_h, \nabla_w \cdot \mathbf{v})+s_1(e_h^{\mathbf{u}}, \mathbf{v})
=&l_1(\mathbf{u}, \mathbf{v})+l_2(p, \mathbf{v})+s_1(Q_h\mathbf{u}, \mathbf{v})\\
&+l_3(\mathbf{u}, \mathbf{v})+l_4(\mathbf{u}, \mathbf{v}),\\
(\nabla_w \cdot e_h^{\mathbf{u}}, q_h )+s_2(e_h^p, q_h)=& l_5(\mathbf{u}, q_h)+s_2(Q_h^p p, q_h),
\end{aligned}
\end{equation*}
for all $\mathbf{v}\in \mathbf{V}_h^0$ and $q_h\in W_h$.
\end{lemma}
\begin{proof}
For $\mathbf{v}\in \mathbf{V}_h^0$, from Lemma \ref{lem:identity} we have
\begin{equation}\label{eq:spid1}
\begin{aligned}
(A\nabla_wQ_h\mathbf{u}, \nabla_w \mathbf{v})=&(\nabla_wQ_h\mathbf{u}, \bm{Q}_sA\nabla_w \mathbf{v})+(\nabla_wQ_h\mathbf{u}, (I-\bm{Q}_s)A\nabla_w \mathbf{v})\\
=&(\nabla \mathbf{u}, \bm{Q}_sA\nabla_w \mathbf{v})+((I-Q_0)\mathbf{u}, \nabla \cdot \bm{Q}_sA\nabla_w \mathbf{v})\\
&+(\nabla_wQ_h\mathbf{u}, (I-\bm{Q}_s)A\nabla_w \mathbf{v}).
\end{aligned}
\end{equation}
By using Lemma \ref{lem:identity}, we may rewrite the first term of the right-hand side of \eqref{eq:spid1} as follows on the element $T$	
\begin{equation*}
\begin{aligned}
&(\nabla \mathbf{u}, \bm{Q}_sA\nabla_w \mathbf{v})_T=(\bm{Q}_s\nabla \mathbf{u}, A\nabla_w \mathbf{v})_T\\
=&-(\nabla\cdot A \bm{Q}_s\nabla\mathbf{u}, \mathbf{v}_0)_T+ \langle A\bm{Q}_s\nabla\mathbf{u}\cdot\mathbf{n}, \mathbf{v}_b\rangle_{\partial T}\\
=& (\bm{Q}_s\nabla\mathbf{u}, A\nabla\mathbf{v}_0)_T+ \langle A\bm{Q}_s\nabla\mathbf{u}\cdot\mathbf{n}, \mathbf{v}_b-\mathbf{v}_0\rangle_{\partial T}\\
=& (A\nabla\mathbf{u}, \nabla\mathbf{v}_0)_T+((\bm{Q}_s-I)\nabla\mathbf{u}, A\nabla\mathbf{v}_0)_T+ \langle A\bm{Q}_s\nabla\mathbf{u}\cdot\mathbf{n}, \mathbf{v}_b-\mathbf{v}_0\rangle_{\partial T}\\
=& -(\nabla\cdot(A\nabla\mathbf{u}), \mathbf{v}_0)_T+\langle A\nabla\mathbf{u}\cdot\mathbf{n}, \mathbf{v}_0\rangle_{\partial T}
+ \langle A\bm{Q}_s\nabla\mathbf{u}\cdot\mathbf{n}, \mathbf{v}_b-\mathbf{v}_0\rangle_{\partial T}\\
&+((\bm{Q}_s-I)\nabla\mathbf{u}, A\nabla\mathbf{v}_0)_T.
\end{aligned}
\end{equation*}
Thus, using $\sum_{T\in\mathcal{T}_h}\langle A\nabla\mathbf{u}\cdot\mathbf{n}, \mathbf{v}_b\rangle_{\partial T}=0$ we arrive at
\begin{equation*}
\begin{aligned}
(\nabla \mathbf{u}, \bm{Q}_s A\nabla_w \mathbf{v})=&
-(\nabla\cdot(A\nabla\mathbf{u}), \mathbf{v}_0)- \langle (I-\bm{Q}_s)A\nabla\mathbf{u}\cdot\mathbf{n}, \mathbf{v}_b-\mathbf{v}_0\rangle\\
&+((\bm{Q}_s-I)A\nabla\mathbf{u}, \nabla\mathbf{v}_0).
\end{aligned}
\end{equation*}
Substituting the above identity into \eqref{eq:spid1} leads to
\begin{equation}\label{eq:spid2}
\begin{aligned}
(A\nabla_wQ_h\mathbf{u}, \nabla_w \mathbf{v})
=&-(\nabla\cdot(A\nabla\mathbf{u}), \mathbf{v}_0)- \langle (I-\bm{Q}_s)A\nabla\mathbf{u}\cdot\mathbf{n}, \mathbf{v}_b-\mathbf{v}_0\rangle\\
&+((\bm{Q}_s-I)A\nabla\mathbf{u}, \nabla\mathbf{v}_0)+((I-Q_0)\mathbf{u}, \nabla \cdot \bm{Q}_sA\nabla_w \mathbf{v})\\
&+(\nabla_wQ_h\mathbf{u}, (I-\bm{Q}_s)A\nabla_w \mathbf{v}).
\end{aligned}
\end{equation}
From the condition $k-1\leq n\leq m$ and the definition of weak divergence \eqref{weak-div} we have
\begin{equation}\label{eq:sec_ter}
\begin{aligned}
(Q_h^p p, \nabla_w \cdot\mathbf{v})&=( p, \nabla \cdot\mathbf{v}_0)+\sum_{T\in\mathcal{T}_h} \langle Q_h^p p, (\mathbf{v}_b- \mathbf{v}_0)\cdot \mathbf{n}\rangle_{\partial T}\\
&= -(\nabla p, \mathbf{v}_0)+\sum_{T\in\mathcal{T}_h}\langle p, \mathbf{v}_0\cdot\mathbf{n} \rangle_{\partial T}+\langle Q_h^p p, (\mathbf{v}_b- \mathbf{v}_0)\cdot \mathbf{n}\rangle_{\partial T}\\
&= -(\nabla p, \mathbf{v}_0)-\sum_{T\in\mathcal{T}_h} \langle (I-Q_h^p) p, (\mathbf{v}_b- \mathbf{v}_0)\cdot \mathbf{n}\rangle_{\partial T},
\end{aligned}
\end{equation}
where we have used $\sum_{T\in\mathcal{T}_h}\langle p, \mathbf{v}_b\cdot\mathbf{n}\rangle_{\partial T}=\sum_{e\in\mathcal{E}_I}\langle \llbracket p \rrbracket , \mathbf{v}_b\cdot\mathbf{n}_e\rangle_{e}=0$.

Next, for $q_h\in W_h$, by using the definition of weak divergence \eqref{weak-div} and $n\leq \min\{m, k+1\}$, we obtain
\begin{equation}\label{eq:thir_ter}
\begin{aligned}
(\nabla_w \cdot Q_h\mathbf{u}, q_h)=(\nabla\cdot\mathbf{u}, q_h)+\langle(Q_b-I)\mathbf{u}\cdot\mathbf{n},q_h\rangle
\end{aligned}
\end{equation}
Hence, combining \eqref{eq:spid2}, \eqref{eq:sec_ter}, \eqref{eq:thir_ter} with \eqref{stokeproblem} yields
\begin{equation}\label{eq:proj_problem}
\begin{aligned}
&(A\nabla_w Q_h {\mathbf{u}}, \nabla_w \mathbf{v})-(Q_h^p p, \nabla_w \cdot \mathbf{v})+s_1(Q_h {\mathbf{u}}, \mathbf{v})\\
=&(\mathbf{f}, \mathbf{v}_0)+s_1(Q_h\mathbf{u}, \mathbf{v})
+((\bm{Q}_s-I)A\nabla\mathbf{u}, \nabla\mathbf{v}_0)\\
& +((I-Q_0)\mathbf{u}, \nabla\cdot A\bm{Q}_s\nabla_w\mathbf{v})+(A\nabla_wQ_h\mathbf{u}, (I-\bm{Q}_s)\nabla_w\mathbf{v})\\
&+\langle (I-Q_h^p)p, (\mathbf{v}_b-\mathbf{v}_0)\cdot\mathbf{n}\rangle+\langle A(I-\bm{Q}_s)\nabla\mathbf{u}\cdot\mathbf{n}, \mathbf{v}_0-\mathbf{v}_b\rangle,\\
&(\nabla_w \cdot Q_h {\mathbf{u}}, q_h )+s_2(Q_h^p p, q_h)\\
= & \langle (Q_b-I)\mathbf{u}\cdot\mathbf{n}, q_h\rangle+s_2(Q_h^p p, q_h).
\end{aligned}
\end{equation}
Finally, subtracting \eqref{sch:WG_FEM} from \eqref{eq:proj_problem} gives rise to the desired error equations.
\end{proof}
\medskip

The following is a useful estimate between the usual gradient and the generalized weak gradient.

\begin{lemma}\label{lem:norm_relative}
There exists a constant $C$ such that  $\Vert\nabla\mathbf{v}_0\Vert \leq C(1+h^{\frac{\gamma-1}{2}})\interleave \mathbf{v}\interleave$ for all $\mathbf{v}\in \mathbf{V}_h$.
\end{lemma}
\begin{proof}
From the definition of the generalized weak gradient \ref{weakgradient} we have $\nabla\mathbf{v}_0=\nabla_w\mathbf{\mathbf{v}}_h-\delta_w\mathbf{v}$.  Thus,
\begin{equation*}
\begin{aligned}
\Vert \nabla\mathbf{v}_0\Vert\leq \Vert\nabla_w\mathbf{v}\Vert+\Vert \delta_w\mathbf{v}\Vert.
\end{aligned}
\end{equation*}
Now, applying the trace inequality to Definition {\color{blue}\ref{weak-grad}} on the element $T$ yields
\begin{equation*}
\begin{aligned}
\Vert \delta_w\mathbf{v}\Vert_T^2&=\langle \mathbf{v}_b-Q_b\mathbf{v}_0, \delta_w\mathbf{v}\rangle_{\partial T}
\leq \Vert \mathbf{v}_b-Q_b\mathbf{v}_0\Vert_{\partial T}\Vert\delta_w\mathbf{v}\Vert_{\partial T}\\
&\leq C h_T^{-\frac{\gamma}{2}}\Vert \mathbf{v}_b-Q_b\mathbf{v}_0\Vert_{\partial T} h_T^{\frac{\gamma-1}{2}}\Vert\delta_w\mathbf{v}\Vert_{ T},
\end{aligned}
\end{equation*}
which leads to $\Vert \delta_w\mathbf{v}\Vert\leq C  h^{\frac{\gamma-1}{2}} s_1(\mathbf{v}, \mathbf{v})^{\frac{1}{2}}\leq Ch^{\frac{\gamma-1}{2}}\interleave \mathbf{v}\interleave$.
\end{proof}

\medskip
The following are concerned with estimates for the linear functionals $l_i$.
\begin{lemma}\label{le:err_ter_est}
Assume that $\mathbf{u}\in [H^{\alpha}(\Omega)]^d$ where $\alpha=\max\{s+2,j, k+1\}$ and $p\in H^{n+1}(\Omega)$. There hold the following estimates:
\begin{equation*}
\begin{aligned}
\left|l_1(\mathbf{u}, \mathbf{v})\right|&\leq Ch^{s+1}(1+h^{\frac{\gamma-1}{2}})\Vert \mathbf{u}\Vert_{s+2}\interleave \mathbf{v}\interleave,\\
\left|l_2(p, \mathbf{v})\right|&\leq Ch^{n+1}(1+h^{\frac{\gamma-1}{2}})\Vert p\Vert_{n+1}\interleave \mathbf{v}\interleave,\\
\left|l_3(\mathbf{u}, \mathbf{v})\right|&\leq C\Big(h^{s+1}(1+h^{\frac{\gamma-1}{2}})\Vert \mathbf{u}\Vert_{s+2}+h^k\Vert \mathbf{u}\Vert_{k+1}\Big)\interleave \mathbf{v}\interleave,\\
\left|l_4(\mathbf{u}, \mathbf{v})\right|&\leq C\Big(h^{s+1}\Vert \mathbf{u}\Vert_{s+2}+h^k\Vert \mathbf{u}\Vert_{k+1}\Big)\interleave \mathbf{v}\interleave,\\
\left|l_5(\mathbf{u}, q_h)\right|&\leq Ch^{j}h^{\frac{\beta+1}{2}}\Vert \mathbf{u}\Vert_{j+1}s_2(q_h, q_h)^{\frac{1}{2}},\\
\left|s_1(Q_h\mathbf{u}, \mathbf{v})\right|&\leq Ch^kh^{\frac{1-\gamma}{2}}\Vert \mathbf{u}\Vert_{k+1}\interleave \mathbf{v}\interleave,\\
\left|s_2(Q_h^p p, q_h)\right|&\leq Ch^{n+1}\Vert p\Vert_{n+1} h^{-\frac{\beta+1}{2}}s_2(q_h, q_h)^{\frac{1}{2}},
\end{aligned}
\end{equation*}
for all $\mathbf{v}\in \mathbf{V}_h^0$ and $q_h\in W_h$.
\end{lemma}
\begin{proof}
From the triangle inequality, Lemma \ref{lem:norm_relative}, and the trace inequality we obtain
\begin{equation*}
\begin{aligned}
\left|l_1(\mathbf{u}, \mathbf{v})\right|\leq & C\sum_{T\in\mathcal{T}_h}\Vert (I-\bm{Q}_s)A\nabla\mathbf{u}\Vert_{\partial T} \Vert \mathbf{v}_0- \mathbf{v}_b\Vert_{\partial T}\\
\leq  & C(1+h^{\frac{\gamma-1}{2}})\Vert (I-\bm{Q}_s)A\nabla\mathbf{u}\Vert \interleave \mathbf{v}\interleave\\
\leq & Ch^{s+1}\Vert\mathbf{u}\Vert_{s+2} (1+h^{\frac{\gamma-1}{2}})\interleave \mathbf{v}\interleave,\\
\left|l_2(p, \mathbf{v})\right|\leq & \sum_{T\in\mathcal{T}_h}\Vert (I-Q_h^p) p\Vert_{\partial T} \Vert \mathbf{v}_0- \mathbf{v}_b\Vert_{\partial T}\\
\leq & Ch^{n+1}\Vert p\Vert_{n+1}(1+h^{\frac{\gamma-1}{2}})\interleave \mathbf{v}\interleave,
\end{aligned}
\end{equation*}
where we have used the following estimate
\begin{equation*}
\begin{aligned}
\Vert \mathbf{v}_0- \mathbf{v}_b\Vert_{\partial T}\leq & \Vert \mathbf{v}_0- Q_b\mathbf{v}_0\Vert_{\partial T}+\Vert Q_b\mathbf{v}_0- \mathbf{v}_b\Vert_{\partial T}\\
\leq & Ch^{\frac{1}{2}}\Vert\nabla\mathbf{v}_0\Vert_T+h^{\frac{\gamma}{2}}s_1(\mathbf{v}, \mathbf{v})^{\frac{1}{2}}\\
\leq & Ch^{\frac{1}{2}}(1+h^{\frac{\gamma-1}{2}})\interleave \mathbf{v}\interleave.
\end{aligned}
\end{equation*}
Now, using the property of projection and the inverse inequality,  we arrive at
\begin{equation*}
\begin{aligned}
\left|l_3(\mathbf{u}, \mathbf{v})\right|&\leq C\Big(\Vert(\bm{Q}_s-I)A\nabla\mathbf{u}\Vert \Vert\nabla\mathbf{v}_0\Vert
+\Vert(I-Q_0)\mathbf{u}\Vert \Vert \nabla\cdot \bm{Q}_sA\nabla_w\mathbf{v}\Vert\Big)\\
&\leq C\Big(h^{s+1}(1+h^{\frac{\gamma-1}{2}})\Vert \mathbf{u}\Vert_{s+2}+h^k\Vert \mathbf{u}\Vert_{k+1}\Big)\interleave \mathbf{v}\interleave.
\end{aligned}
\end{equation*}
As to $l_4$, it follows from the definition of weak gradient {\color{blue}\eqref{weakgradient}} that
\begin{equation*}
\begin{aligned}
\left|l_4(\mathbf{u}, \mathbf{v})\right|&=|(\nabla Q_0\mathbf{u}+\delta_w Q_h\mathbf{u}, (I-\bm{Q}_s)A\nabla_w\mathbf{v})|\\
&\leq |((I-\bm{Q}_s)\nabla Q_0\mathbf{u}, A\nabla_w\mathbf{v})|
+C\Vert \delta_w Q_h\mathbf{u} \Vert \Vert(I-\bm{Q}_s)A\nabla_w\mathbf{v} \Vert  \\
&\leq C\Vert (I-\bm{Q}_s)\nabla Q_0\mathbf{u} \Vert \Vert\nabla_w\mathbf{v}\Vert+ \Vert \delta_w Q_h\mathbf{u} \Vert \Vert(I-\bm{Q}_s)A\nabla_w\mathbf{v}\Vert\\
&\leq Ch^{s+1}\Vert Q_0\mathbf{u} \Vert_{s+2} \Vert\nabla_w\mathbf{v}\Vert+ C h^{k}\Vert \mathbf{u}\Vert_{k+1}\Vert\nabla_w\mathbf{v}\Vert\\
&\leq C(h^{s+1}\Vert \mathbf{u}\Vert_{s+2}+h^k\Vert \mathbf{u}\Vert_{k+1})\Vert\nabla_w\mathbf{v}\Vert.
\end{aligned}
\end{equation*}
Here, we have used the following estimate
\begin{equation}\label{eq:delt_est}
\begin{aligned}
\Vert \delta_w Q_h\mathbf{u}\Vert_T^2&=\langle Q_b(I-Q_0)\mathbf{u},  \delta_w Q_h\mathbf{u}\rangle_{\partial T}\\
&\leq \Vert (I-Q_0)\mathbf{u}\Vert_{\partial T}\Vert\delta_w Q_h\mathbf{u}\Vert_{\partial T}\\
&\leq Ch^{k}\Vert \mathbf{u}\Vert_{k+1} \Vert \delta_w Q_h\mathbf{u}\Vert_T.
\end{aligned}
\end{equation}
Next, using the trace inequality and the property of $L^2$ projections, we deduce that
\begin{equation*}
\begin{aligned}
|s_1(Q_h\mathbf{u}, \mathbf{v})|&\leq \sum_{T\in\mathcal{T}_h}h_T^{-\gamma}|\langle(I-Q_0)\mathbf{u}, \mathbf{v}_b-Q_b\mathbf{v}_0\rangle_{\partial T}|\\
&\leq \sum_{T\in\mathcal{T}_h}h_T^{\frac{-\gamma}{2}} \Vert (I-Q_0)\mathbf{u} \Vert_{\partial T} h_T^{\frac{-\gamma}{2}} \Vert \mathbf{v}_b-Q_b\mathbf{v}_0 \Vert_{\partial T}\\
&\leq Ch^k\Vert \mathbf{u}\Vert_{k+1} h^{\frac{1-\gamma}{2}} s_1(\mathbf{v}, \mathbf{v})^{\frac{1}{2}}.
\end{aligned}
\end{equation*}
Owing to the property of the $L^2$ projections, one finds
\begin{equation*}
\begin{aligned}
|s_2(Q_h^p p, q_h)|=&\left|\sum_{e\in\mathcal{E}_h}h_e^{-\beta}\langle \llbracket Q^p_h p \rrbracket,   \llbracket q_h \rrbracket \rangle_e\right|
=\left|\sum_{e\in\mathcal{E}_h}h_e^{-\beta}\langle \llbracket (I - Q^p_h) p \rrbracket,   \llbracket q_h \rrbracket \rangle_e\right|\\
\leq&\sum_{e\in\mathcal{E}_h}h_e^{-\frac{\beta}{2}}\Vert \llbracket (I - Q^p_h) p\rrbracket \Vert_e  h_e^{-\frac{\beta}{2}} \Vert  \llbracket q_h \rrbracket \Vert_e\\
\leq&Ch^{n+1}\Vert p\Vert_{n+1} h^{-\frac{\beta+1}{2}}s_2(q_h, q_h)^{\frac{1}{2}},
\end{aligned}
\end{equation*}
where we have used the identity $\sum_{e\in\mathcal{E}_h}h^{-\beta}\langle \llbracket p \rrbracket,   \llbracket q_h \rrbracket \rangle_e=0$.

Finally,  it follows from the trace inequality that
\begin{equation*}
\begin{aligned}
|l_5(\mathbf{u}, q_h)|&=\left|\sum_{T\in\mathcal{T}_h}\langle (Q_b-I)\mathbf{u}\cdot\mathbf{n}, q_h\rangle_{\partial T}\right|
=\left|\sum_{e\in\mathcal{E}_I}\langle (Q_b-I)\mathbf{u}\cdot\mathbf{n}_e, \llbracket q_h\rrbracket \rangle_{e}\right|\\
&\leq\sum_{e\in\mathcal{E}_I}h_e^{\frac{1}{2}}\Vert (I-Q_b)\mathbf{u}\Vert_e h_e^{\frac{\beta-1}{2}} h_e^{\frac{-\beta}{2}}\Vert\llbracket q_h\rrbracket \Vert_{e}\\
&\leq Ch^{j}\Vert \mathbf{u}\Vert_{j+1}h^{\frac{\beta+1}{2}}s_2(q_h, q_h)^{\frac{1}{2}}.
\end{aligned}
\end{equation*}
This completes the proof of the lemma.
\end{proof}

\section{Error Analysis}\label{sec:5}
In this section, we derive an error estimate for the velocity and pressure approximations arising from the scheme \eqref{sch:WG_FEM} with elements of arbitrary polynomial combinations.
In particular, for the case of $n\leq j$, we show that the stabilizer $s_2$ is no longer necessary for the scheme to be well-posed.

\subsection{Error estimate with $\mu > 0$}\label{error_case1} For simplicity of presentation, we assume $\mu=1$.
\begin{theorem}\label{err_esti_energy}
Let $(\mathbf{u}, p)$ be the exact solution of the Stokes problem \eqref{stokeproblem} and $(\mathbf{u}_h, p_h)$ the numerical solution arising from the generalized weak Galerkin scheme \eqref{sch:WG_FEM}.
Assume $\mathbf{u}\in [H^{\alpha}(\Omega)]^d$ with $\alpha=\max\{s+2, j+1, k+1\}$, $p\in H^{n+1}(\Omega)$, and $k-1\leq n\leq\min\{m, k+1\}$. Then, there exists a constant $C$ such that
\begin{equation*}
\begin{aligned}
\interleave Q_h\mathbf{u}-\mathbf{u}_h \interleave \leq& C\Big(h^{s+1}(1+h^{\frac{\gamma-1}{2}})\Vert \mathbf{u}\Vert_{s+2}
+h^k(1+h^{\frac{1-\gamma}{2}})\Vert \mathbf{u}\Vert_{k+1}\\
&+h^{j}h^{\frac{\beta+1}{2}}\Vert \mathbf{u}\Vert_{j+1}+h^{n+1}(1+h^{\frac{\gamma-1}{2}}+h^{-\frac{\beta+1}{2}})\Vert p\Vert_{n+1}{\color{blue}\Big)},\\
\Vert Q_h^p p-p_h \Vert \leq& C(1+h^{\frac{1-\gamma}{2}}+h^{\frac{\beta+1}{2}})\Big(h^{s+1}(1+h^{\frac{\gamma-1}{2}})\Vert \mathbf{u}\Vert_{s+2}
+h^k(1+h^{\frac{1-\gamma}{2}})\Vert \mathbf{u}\Vert_{k+1}\\
&+h^{j}h^{\frac{\beta+1}{2}}\Vert \mathbf{u}\Vert_{j+1}+h^{n+1}(1+h^{\frac{\gamma-1}{2}}+h^{-\frac{\beta+1}{2}})\Vert p\Vert_{n+1}\Big).
\end{aligned}
\end{equation*}
Moreover, the optimal order of convergence can be achieved when $\gamma=1$ and $\beta=-1$.
\end{theorem}
\begin{proof} By taking $\mathbf{v}=e_h^{\mathbf{u}}$ and $q_h=e_h^p$ in the error equations in Lemma \ref{lem:errorequ},  we have
\begin{equation*}
\begin{aligned}
A(\nabla_w e_h^{\mathbf{u}}, \nabla_w e_h^{\mathbf{u}})-(e^p_h, \nabla_w \cdot e_h^{\mathbf{u}})+s_1(e_h^{\mathbf{u}}, e_h^{\mathbf{u}})
=&l_1(\mathbf{u}, e_h^{\mathbf{u}})+l_2(p, e_h^{\mathbf{u}})+s_1(Q_h\mathbf{u}, e_h^{\mathbf{u}})\\
&+l_3(\mathbf{u}, e_h^{\mathbf{u}})+l_4(\mathbf{u}, e_h^{\mathbf{u}}), \\
(\nabla_w \cdot e_h^{\mathbf{u}}, e_h^p )+s_2(e_h^p, e_h^p)=& l_5(\mathbf{u}, e_h^p)+s_2(Q_h^p p, e_h^p).
\end{aligned}
\end{equation*}
Summing the above two equations and using Lemma \ref{le:err_ter_est}, we arrive at
\begin{equation}\label{eq:err1}
\begin{aligned}
&\interleave e_h^{\mathbf{u}}\interleave^2 +s_2(e_h^p, e_h^p)
\leq  C\Big(h^{s+1}(1+h^{\frac{\gamma-1}{2}})\Vert \mathbf{u}\Vert_{s+2}\\
&+h^{n+1}(1+h^{\frac{\gamma-1}{2}})\Vert p\Vert_{n+1}
+h^k(1+h^{\frac{1-\gamma}{2}})\Vert \mathbf{u}\Vert_{k+1}\Big)\interleave e_h^{\mathbf{u}}\interleave\\
&+C\Big(h^{j}h^{\frac{\beta+1}{2}}\Vert \mathbf{u}\Vert_{j+1}+h^{n+1}h^{-\frac{\beta+1}{2}}\Vert p\Vert_{n+1}\Big) s_2(e_h^p, e_h^p)^{\frac{1}{2}}.
\end{aligned}
\end{equation}
Next, by using the {\em inf-sup} condition \eqref{inf_supconditon}, Lemma \ref{lem:errorequ}, and  Lemma \ref{le:err_ter_est}, we obtain
\begin{equation*}
\begin{aligned}
&\frac{1}{2}\Vert e^q_h\Vert^2-C\sum_{e\in\mathcal{E}_{I}}h_e\Vert  \llbracket e^q_h \rrbracket \Vert_e^2\leq(e^p_h, \nabla_w \cdot \mathbf{v})=A(\nabla_w e_h^{\mathbf{u}}, \nabla_w \mathbf{v})+s_1(e_h^{\mathbf{u}}, \mathbf{v})\\
&-\Big(l_1(\mathbf{u}, \mathbf{v})+l_2(p, \mathbf{v})+s_1(Q_h\mathbf{u}, \mathbf{v})
+l_3(\mathbf{u}, \mathbf{v})+l_4(\mathbf{u}, \mathbf{v})\Big)\\
\leq &C(\interleave e_h^{\mathbf{u}}\interleave +h^{s+1}(1+h^{\frac{\gamma-1}{2}})\Vert \mathbf{u}\Vert_{s+2}+h^{n+1}(1+h^{\frac{\gamma-1}{2}})\Vert p\Vert_{n+1}\\
&+h^k(1+h^{\frac{1-\gamma}{2}})\Vert \mathbf{u}\Vert_{k+1})
\interleave \mathbf{v}\interleave
\end{aligned}
\end{equation*}
for some $\mathbf{v}$ satisfying \eqref{inf_supconditon}. It follows from the Cauchy-Schwarz inequality that
\begin{equation*}
\begin{aligned}
\frac{1}{2}\Vert e^q_h\Vert^2 \leq &C(\interleave e_h^{\mathbf{u}}\interleave +h^{s+1}(1+h^{\frac{\gamma-1}{2}})\Vert \mathbf{u}\Vert_{s+2}
+h^{n+1}(1+h^{\frac{\gamma-1}{2}})\Vert p\Vert_{n+1}\\
&+h^k(1+h^{\frac{1-\gamma}{2}})\Vert \mathbf{u}\Vert_{k+1}) (1+h^{\frac{1-\gamma}{2}})
\Vert e^q_h\Vert +Ch^{\beta+1}s_2(e_h^p, e_h^p)\\
\leq&C(1+h^{\frac{1-\gamma}{2}})^2\interleave e_h^{\mathbf{u}}\interleave^2 +\frac{1}{4}\Vert e^q_h\Vert^2+Ch^{\beta+1}s_2(e_h^p, e_h^p).
\end{aligned}
\end{equation*}
Therefore, we have
\begin{equation}\label{eq:err2}
\begin{aligned}
\Vert e^q_h\Vert \leq&C(1+h^{\frac{1-\gamma}{2}})\interleave e_h^{\mathbf{u}}\interleave +Ch^{\frac{\beta+1}{2}}s_2(e_h^p, e_h^p)^{\frac{1}{2}}.
\end{aligned}
\end{equation}
Finally, combining \eqref{eq:err1} with \eqref{eq:err2} completes the proof of the theorem.
\end{proof}

\begin{remark}\label{re:errest}
For the case of $\gamma=1$ and $\beta=-1$, the error estimate in Theorem  \ref{err_esti_energy} can be rewritten as
\begin{equation*}
\begin{aligned}
\interleave Q_h\mathbf{u}-\mathbf{u}_h \interleave+\Vert Q_h^p p-p_h \Vert\leq& C(h^{s+1}\Vert \mathbf{u}\Vert_{s+2}+h^k\Vert \mathbf{u}\Vert_{k+1}+h^{j}\Vert \mathbf{u}\Vert_{j+1}\\
&+h^{n+1}\Vert p\Vert_{n+1}),
\end{aligned}
\end{equation*}
which clearly represents an optimal order of error estimate for the generalized weak Galerkin scheme \eqref{sch:WG_FEM}.
\end{remark}

\medskip
To estimate the velocity error $e_h^{\mathbf{u}}$ in the $L^2$ norm,  we shall employ the usual duality argument by considering an auxiliary problem that seeks $(\varphi, \xi)$ such that
\begin{equation}\label{pro:dual}
\begin{aligned}
-\nabla\cdot A\nabla\varphi+\nabla\xi&=e_0^{\mathbf{u}}, ~~\mbox{in} ~\Omega,\\
\nabla\cdot\varphi&=0,~~~\mbox{in}~ \Omega,\\
\varphi&=0,~~~\mbox{on}~ \partial\Omega.
\end{aligned}
\end{equation}
Assume that the problem \eqref{pro:dual} is $H^2$-regular; i.e.  for $e_0^{\mathbf{u}}\in [L^2(\Omega)]^d$, the solution $(\varphi, \xi)$ is so regular that $\varphi\in[H^2(\Omega)]^d$ and $ \xi\in H^1(\Omega)$ and satisfies
\begin{equation}\label{asu:regular}
\begin{aligned}
\Vert \varphi\Vert_2+ \Vert \xi\Vert_1\leq C \Vert e_0^{\mathbf{u}}\Vert.
\end{aligned}
\end{equation}

\begin{theorem}\label{err_esti_L2}
Let $(\mathbf{u}, p)$ be the exact solution of the Stokes problem \eqref{stokeproblem} and $(\mathbf{u}_h, p_h)$ the numerical solution arising from the generalized weak Galerkin method \eqref{sch:WG_FEM}.
Assume $\mathbf{u}\in [H^{\alpha}(\Omega)]^d$ where  $\alpha=\max\{s+2, j+1, k+1\}$, $p\in H^{n+1}(\Omega)$, $k\geq 1$, $k-1\leq n\leq\min\{m, k+1\}$, and either $j\ge n$ or $j\geq 1$. Assume that the problem \eqref{pro:dual} is $H^2$-regular. Then, there exists a constant $C$ such that
\begin{equation*}
\begin{aligned}
 \Vert Q_0\mathbf{u}-\mathbf{u}_0 \Vert \leq& C(1+h^{\frac{\gamma-1}{2}}+h^{\frac{1-\gamma}{2}}+h^{\frac{\beta+1}{2}})
 \Big(h^{s+2}(1+h^{\frac{\gamma-1}{2}})\Vert \mathbf{u}\Vert_{s+2}\\
 &+h^{j+1}(1+h^{\frac{\beta+1}{2}})\Vert \mathbf{u}\Vert_{j+1}+h^{k+1}(1+h^{\frac{1-\gamma}{2}}+h^{1-\gamma})\Vert \mathbf{u}\Vert_{k+1}\\
&+h^{n+2}(1+h^{\frac{\gamma-1}{2}}+h^{-\frac{\beta+1}{2}})\Vert p\Vert_{n+1}\Big).
\end{aligned}
\end{equation*}
Moreover, the optimal order of convergence can be achieved when $\gamma=1$ and $\beta=-1$.
\end{theorem}
\begin{proof}
Similarly to the first equation of \eqref{eq:proj_problem}, we have for any $\mathbf{v}\in \mathbf{V}_h^0$
\begin{equation}\label{eq:proj_dual_1}
\begin{aligned}
&(A\nabla_w Q_h \varphi, \nabla_w \mathbf{v})-(Q_h^p \xi, \nabla_w \cdot \mathbf{v})+s_1(Q_h \varphi, \mathbf{v})
=(e_0^{\mathbf{u}}, \mathbf{v}_0)+s_1(Q_h\varphi, \mathbf{v})\\
&+(A\nabla_wQ_h\varphi, (I-\bm{Q}_s)\nabla_w\mathbf{v})+(A(\bm{Q}_s-I)\nabla\varphi, \nabla\mathbf{v}_0)\\
&+((I-Q_0)\varphi, \nabla\cdot \bm{Q}_sA\nabla_w\mathbf{v})
+\sum_{T\in\mathcal{T}_h}\langle (I-Q_h^p)\xi, (\mathbf{v}_b-\mathbf{v}_0)\cdot\mathbf{n}\rangle_{\partial T}\\
&+\sum_{T\in\mathcal{T}_h}\langle A(I-\bm{Q}_s)\nabla\varphi\cdot\mathbf{n}, \mathbf{v}_0-\mathbf{v}_b\rangle_{\partial T}\\
&:=(e_0^{\mathbf{u}}, \mathbf{v}_0)+L_{\varphi, \xi}(\mathbf{v}).
\end{aligned}
\end{equation}
By taking $\mathbf{v}=e_h^{\mathbf{u}}$ in \eqref{eq:proj_dual_1}, we arrive at
\begin{equation}\label{eq:proj_dual_2}
\begin{aligned}
&\Vert e_0^{\mathbf{u}}\Vert^2 =(A\nabla_w Q_h \varphi, \nabla_w e_h^{\mathbf{u}})-(Q_h^p \xi, \nabla_w \cdot e_h^{\mathbf{u}})+s_1(Q_h \varphi, e_h^{\mathbf{u}})
-L_{\varphi, \xi}(e_h^{\mathbf{u}}).
\end{aligned}
\end{equation}

Next, choosing $q_h=Q_h^p \xi$ in the second equation  of Lemma  \ref{lem:errorequ} leads to

\begin{equation}\label{eq:proj_dual_3}
\begin{aligned}
(\nabla_w \cdot e_h^{\mathbf{u}}, Q_h^p \xi)+s_2(e_h^p,Q_h^p \xi)&= l_5(\mathbf{u}, Q_h^p \xi)+s_2(Q_h^p p, Q_h^p \xi).
\end{aligned}
\end{equation}
From the definition of weak divergence and the property of $L^2$ projections one finds
\begin{equation}\label{eq:proj_dual_4}
\begin{aligned}
(\nabla_w \cdot Q_h \varphi, e_h^p )= l_5(\varphi, e_h^p).
\end{aligned}
\end{equation}
Substituting \eqref{eq:proj_dual_3}, \eqref{eq:proj_dual_4} and Lemma \ref{lem:errorequ} into \eqref{eq:proj_dual_2}, we deduce
\begin{equation}\label{eq:proj_dual_final}
\begin{aligned}
\Vert e_0^{\mathbf{u}}\Vert^2 =&(A\nabla_w Q_h \varphi, \nabla_w e_h^{\mathbf{u}})-(\nabla_w \cdot Q_h \varphi, e_h^p )+s_1(Q_h \varphi, e_h^{\mathbf{u}})\\
&+(\nabla_w \cdot Q_h \varphi, e_h^p )-(Q_h^p \xi, \nabla_w \cdot e_h^{\mathbf{u}})-L_{\varphi, \xi}(e_h^{\mathbf{u}})\\
=&L_{\mathbf{u}, p}(Q_h \varphi)- l_5(\mathbf{u}, Q_h^p \xi)+l_5(\varphi, e_h^p)-L_{\varphi, \xi}(e_h^{\mathbf{u}})\\
&-s_2(Q_h^p p, Q_h^p \xi)+s_2(e_h^p,Q_h^p \xi),
\end{aligned}
\end{equation}
where
$$
L_{\mathbf{u}, p}(Q_h \varphi)=l_1(\mathbf{u}, Q_h \varphi )+l_2(p, Q_h \varphi )+l_3(\mathbf{u}, Q_h \varphi )+l_4(\mathbf{u}, Q_h \varphi )+s_1(Q_h\mathbf{u}, Q_h \varphi).
$$
According to Lemma \ref{le:err_ter_est} and $k\geq 1$, we get
\begin{equation}\label{eq:dual_est_1}
\begin{aligned}
\left|L_{\varphi, \xi}(e_h^{\mathbf{u}})\right|\leq Ch(1+h^{\frac{\gamma-1}{2}}+h^{\frac{1-\gamma}{2}})(\Vert \varphi\Vert_2+\Vert \xi\Vert_1)\interleave e_h^{\mathbf{u}}\interleave.
\end{aligned}
\end{equation}
To estimate $L_{\mathbf{u}, p}(Q_h \varphi)$ for $k\geq 1$, we use the definition of the weak gradient {\color{blue}\eqref{weakgradient}} and the property of $L^2$ projections to obtain
\begin{equation*}
\begin{aligned}
\Vert \delta_wQ_h\varphi\Vert^2&=\langle Q_b(I-Q_0)\varphi, \delta_wQ_h\varphi\rangle\leq Ch^{-1}\Vert(I-Q_0)\varphi\Vert \Vert \delta_wQ_h\varphi\Vert\\
&\leq Ch\Vert\varphi\Vert_2 \Vert \delta_wQ_h\varphi \Vert.
\end{aligned}
\end{equation*}
Hence, $\Vert \delta_wQ_h\varphi\Vert \leq Ch\Vert\varphi\Vert_2 $. In virtue of the triangle inequality and the inverse inequality,
 \begin{equation}\label{eq:L2dual_est_1}
\begin{aligned}
\Vert \nabla_wQ_h\varphi\Vert_1\leq \Vert \nabla Q_0\varphi\Vert_1+\Vert \delta_wQ_h\varphi\Vert_1
\leq C\Vert\varphi\Vert_2+Ch^{-1}\Vert \delta_wQ_h\varphi\Vert\leq C \Vert\varphi\Vert_2.
\end{aligned}
\end{equation}
Now from the trace inequality and the property of $L^2$ projections we have
 \begin{equation}\label{eq:L2dual_est_2}
\begin{aligned}
|l_1(\mathbf{u}, Q_h \varphi)|&=\left|\sum_{T\in\mathcal{T}_h}\langle A(I-\mathbf{Q}_s)\nabla\mathbf{u}\cdot\mathbf{n}, Q_0\varphi-Q_b\varphi\rangle_{\partial T}\right|\\
&\leq C\sum_{T\in\mathcal{T}_h}\Vert (I-\mathbf{Q}_s)\nabla\mathbf{u}\Vert_{\partial T}  \Vert Q_0\varphi-\varphi\Vert_{\partial T}\\
&\leq Ch^{s+2}\Vert\mathbf{u}\Vert_{s+2}\Vert\varphi\Vert_2,
\end{aligned}
\end{equation}
where we have used the following identity
 \begin{equation*}
\begin{aligned}
\sum_{T\in\mathcal{T}_h}\langle (I-\mathbf{Q}_s)A\nabla\mathbf{u}\cdot\mathbf{n}, \varphi-Q_b\varphi\rangle_{\partial T}
&=\sum_{T\in\mathcal{T}_h}\langle A\nabla\mathbf{u}\cdot\mathbf{n}, \varphi-Q_b\varphi\rangle_{\partial T}\\
&=\sum_{e\in\mathcal{E}_I}\langle A\nabla\mathbf{u}\cdot\mathbf{n}_e, \llbracket \varphi-Q_b\varphi\rrbracket\rangle_{ e}=0.
\end{aligned}
\end{equation*}
Next,  by taking into account of $k\geq 1$ and either $j\ge n$ or $j\ge 1$, we arrive at
 \begin{equation}\label{eq:L2dual_est_3}
\begin{aligned}
|l_2(p, Q_h \varphi)|&=\left|\sum_{T\in\mathcal{T}_h} \langle (I-Q_h^p)p, (Q_b-Q_0)\varphi\cdot\mathbf{n}\rangle_{\partial T}\right|\\
&\leq Ch^{n+2}\Vert p\Vert_{n+1}\Vert\varphi\Vert_2.
\end{aligned}
\end{equation}

From \eqref{eq:L2dual_est_1} and the property of $L^2$ projections we have
 \begin{equation}\label{eq:L2dual_est_4}
\begin{aligned}
&l_3(\mathbf{u}, Q_h \varphi)=((\mathbf{Q}_s-I)A\nabla\mathbf{u}, \nabla Q_0\varphi) + ((I-Q_0)\mathbf{u}, \nabla\cdot A\mathbf{Q}_s\nabla_wQ_h\varphi) \\
=&((\mathbf{Q}_s-I)A\nabla\mathbf{u}, (I-\mathbf{Q}_s)\nabla Q_0\varphi) + ((I-Q_0)\mathbf{u}, \nabla\cdot A\mathbf{Q}_s\nabla_wQ_h\varphi)\\
\leq& Ch^{s+1}\Vert \mathbf{u}\Vert_{s+2}h\Vert \nabla Q_0\varphi\Vert_1+Ch^{k+1}\Vert \mathbf{u}\Vert_{k+1} \Vert \nabla\cdot \mathbf{Q}_s\nabla_wQ_h\varphi\Vert\\
\leq& C(h^{s+2}\Vert\mathbf{u}\Vert_{s+2}+h^{k+1}\Vert\mathbf{u}\Vert_{k+1})\Vert\varphi\Vert_2.
\end{aligned}
\end{equation}
Using again the definition of weak gradient {\color{blue}\eqref{weakgradient}},  \eqref{eq:delt_est}, and \eqref{eq:L2dual_est_1}, we obtain
 \begin{equation}\label{eq:L2dual_est_5}
\begin{aligned}
|l_4(\mathbf{u}, Q_h \varphi)|=&|(A\nabla_wQ_h\mathbf{u}, (I-\mathbf{Q}_s)\nabla_w Q_h\varphi)| \\
\leq& |((I-\mathbf{Q}_s)A\nabla Q_0\mathbf{u}, (I-\mathbf{Q}_s)\nabla_w Q_h\varphi)| \\
&+Ch^{-1}\Vert Q_b(I-Q_0)\mathbf{u}\Vert \Vert (I-\mathbf{Q}_s)\nabla_w Q_h\varphi\Vert\\
\leq& Ch^{s+1}\Vert \mathbf{u}\Vert_{s+2}h\Vert \nabla_w Q_h\varphi\Vert_1+Ch^{k}\Vert \mathbf{u}\Vert_{k+1} h\Vert \nabla_wQ_h\varphi\Vert_1\\
\leq& C(h^{s+2}\Vert\mathbf{u}\Vert_{s+2}+h^{k+1}\Vert\mathbf{u}\Vert_{k+1})\Vert\varphi\Vert_2.
\end{aligned}
\end{equation}
For the term $s_1( Q_h\mathbf{u}, Q_h \varphi)$, we use the trace inequality and the property of $L^2$ projections to obtain
 \begin{equation}\label{eq:L2dual_est_6}
\begin{aligned}
|s_1( Q_h\mathbf{u}, Q_h \varphi)|&=\left|\sum_{T\in\mathcal{T}_h}h^{-\gamma}\langle Q_b(I-Q_0)\mathbf{u}, Q_b(I-Q_0)\varphi\rangle_{\partial T} \right|\\
&\leq \sum_{T\in\mathcal{T}_h}h^{-\gamma}\Vert  (I-Q_0)\mathbf{u}\Vert_{\partial T} \Vert (I-Q_0)\varphi\Vert_{\partial T}\\
&\leq Ch^{k}h^{1-\gamma}\Vert\mathbf{u}\Vert_{k+1} h\Vert\varphi\Vert_2.
\end{aligned}
\end{equation}
 Finally, from the trace inequality and the property of the $L^2$ projections we have
 \begin{equation}\label{eq:L2dual_est_7}
\begin{aligned}
|l_5(\mathbf{u}, Q_h^p \xi)|&=\left|\sum_{T\mathcal{T}_h}\langle (I-Q_b)\mathbf{u}\cdot\mathbf{n}, (I-Q_b)Q_h^p \xi\rangle_{\partial T}\right|\\
&\leq Ch^{-1}\Vert (I-Q_b)\mathbf{u} \Vert \Vert (I-Q_b)Q_h^p \xi \Vert
\leq Ch^{j+1}\Vert \mathbf{u}\Vert_{j+1} \Vert \xi\Vert_1,\\
|l_5(\varphi, e_h^p)|&=\left|\sum_{T\mathcal{T}_h}\langle (I-Q_b)\varphi\cdot\mathbf{n}, e_h^p \rangle_{\partial T}\right|\\
&\leq Ch \Vert e_h^p \Vert  \Vert\varphi\Vert_2.
\end{aligned}
\end{equation}
 Hence, substituting \eqref{eq:dual_est_1} and \eqref{eq:L2dual_est_2} - \eqref{eq:L2dual_est_7} into \eqref{eq:proj_dual_final}, we infer that
 \begin{equation}\label{eq:err_dual_final}
\begin{aligned}
\Vert e_0^{\mathbf{u}}\Vert^2 \leq& C\Big((1+h^{\frac{\gamma-1}{2}}+h^{\frac{1-\gamma}{2}})\interleave e_h^{\mathbf{u}}\interleave
+\Vert e_h^p \Vert+h^{s+1}\Vert \mathbf{u}\Vert_{s+2}+h^{n+1}\Vert p\Vert_{n+1}\\
&+h^k(1+h^{1-\gamma})\Vert \mathbf{u}\Vert_{k+1}
+h^{j}\Vert \mathbf{u}\Vert_{j+1}\Big) h(\Vert\varphi\Vert_2+\Vert\xi\Vert_1),
\end{aligned}
\end{equation}
which, together with the regularity \eqref{asu:regular} and Theorem \ref{err_esti_energy},
verifies the desired error estimate in $L^2$ for the velocity approximation.
\end{proof}
\medskip

\begin{remark}\label{re:errest2}
For the case of $\gamma=1$ and $\beta=-1$, the error estimate in Theorem  \ref{err_esti_L2}  can be reformulated as follows
\begin{equation*}
\begin{aligned}
\Vert Q_0\mathbf{u}-\mathbf{u}_0 \Vert  &\leq C\Big(h^{s+2}\Vert \mathbf{u}\Vert_{s+2}+h^{n+2}\Vert p\Vert_{n+1}+h^{k+1}\Vert \mathbf{u}\Vert_{k+1}+h^{j+1}\Vert \mathbf{u}\Vert_{j+1}\Big),
\end{aligned}
\end{equation*}
which is clearly of optimal order in terms of the order of polynomials chosen in the finite element scheme.
\end{remark}

\subsection{Error estimate with $j\geq n$ and $\mu=0$}
This subsection is dedicated to the error analysis of the generalized weak Galerkin method when the polynomial order $n$ of the pressure approximation falls in the range $k-1\leq n\leq \min\{m, k+1,j\}$
and the stabilizer $s_2$ vanishes (i.e., $\mu=0$). Recall that from the property of the $L^2$ projections we have $l_5=0$ under this setting.

Since $\mu=0$ or $s_2\equiv 0$, the generalized weak Galerkin method \eqref{sch:WG_FEM} is simplified as follows:
\begin{equation}\label{sch:WG_FEM_1}
\begin{aligned}
(A \nabla_w \mathbf{u}_h, \nabla_w \mathbf{v})-(p_h,\nabla_w\cdot\mathbf{v})+s_1(\mathbf{u}_h, \mathbf{v})&=(\mathbf{f},\mathbf{v}_0), ~~~~~~\forall~ \mathbf{v}\in\mathbf{V}_{h}^0,\\
(\nabla_w\cdot\mathbf{u}_h, q_h)&=0,~~~~~~~~~~~~~\forall~q_h\in W_{h}.
\end{aligned}
\end{equation}
The {\em inf-sup} condition Lemma \ref{lem:inf-sup} can also be reduced to the classical version; namely,
for each $q_h\in W_h$, there exists $\mathbf{v}\in \mathbf{V}_h^0$ such that
\begin{equation}\label{inf_supconditon_1}
\begin{aligned}
b(\mathbf{v}, q_h) &\geq C \Vert q_h\Vert^2,\\
\interleave \mathbf{v}\interleave&\leq C(1+h^{\frac{1-\gamma}{2}}) \Vert q_h\Vert.
\end{aligned}
\end{equation}
Analogue to Lemma \ref{lem:errorequ}, the following result holds true for the error functions for the scheme \eqref{sch:WG_FEM_1}.
\begin{lemma}\label{lem:errorequ_1}
Let $(\mathbf{u}, p)$ be the exact solution of the Stokes problem \eqref{stokeproblem} and $(\mathbf{u}_h, p_h)$ the numerical solution from the generalized weak Galerkin method $\eqref{sch:WG_FEM_1}$. Assume that
$ k-1\leq n\leq\min\{m, k+1, j \}$.
Then we have the following error equations
\begin{equation*}
\begin{aligned}
(A\nabla_w e_h^{\mathbf{u}}, \nabla_w \mathbf{v})-(e^p_h, \nabla_w \cdot \mathbf{v})+s_1(e_h^{\mathbf{u}}, \mathbf{v})
=&l_1(\mathbf{u}, \mathbf{v})+l_2(p, \mathbf{v})+s_1(Q_h\mathbf{u}, \mathbf{v})\\
&+l_3(\mathbf{u}, \mathbf{v})+l_4(\mathbf{u}, \mathbf{v}),\\
(\nabla_w \cdot e_h^{\mathbf{u}}, q_h )=&0,
\end{aligned}
\end{equation*}
for all $\mathbf{v}\in \mathbf{V}_h^0$ and $q_h\in W_h$.
\end{lemma}

Like Theorem \ref{err_esti_energy}, from the {\em inf-sup} condition \eqref{inf_supconditon_1} and  Lemma \ref{le:err_ter_est} we may derive an error estimate for the velocity approximation in the energy norm and the pressure in the $L^2$ norm.
\begin{theorem}\label{err_esti_energy_1}
In addition to the assumptions in Lemma \ref{lem:errorequ_1},  assume $\mathbf{u}\in [H^{\alpha}(\Omega)]^d$ where $\alpha=\max\{s+2, k+1\}$ and $p\in H^{n+1}(\Omega)$. Then, there exists a constant $C$ such that
\begin{equation*}
\begin{aligned}
\interleave Q_h\mathbf{u}-\mathbf{u}_h \interleave \leq& C\Big(h^{s+1}(1+h^{\frac{\gamma-1}{2}})\Vert \mathbf{u}\Vert_{s+2}+h^{n+1}(1+h^{\frac{\gamma-1}{2}})\Vert p\Vert_{n+1}\\
&+h^k(1+h^{\frac{1-\gamma}{2}})\Vert \mathbf{u}\Vert_{k+1}\Big),\\
 \Vert Q_h^p p-p_h \Vert \leq& C(1+h^{\frac{1-\gamma}{2}})\Big(h^{s+1}(1+h^{\frac{\gamma-1}{2}})\Vert \mathbf{u}\Vert_{s+2}+h^{n+1}(1+h^{\frac{\gamma-1}{2}})\Vert p\Vert_{n+1}\\
&+h^k(1+h^{\frac{1-\gamma}{2}})\Vert \mathbf{u}\Vert_{k+1}\Big).
\end{aligned}
\end{equation*}
Moreover, the optimal order of convergence is achieved by choosing $\gamma=1$.
\end{theorem}

\begin{theorem}\label{err_esti_L2_1}
In addition to the assumptions in Theorem \ref{err_esti_energy_1}, assume that $k\ge 1$ and the auxiliary problem \eqref{pro:dual} is $H^2$-regular. Then, there exists a constant $C$ such that
\begin{equation*}
\begin{aligned}
\Vert Q_0\mathbf{u}-\mathbf{u}_0 \Vert  &\leq C(1+h^{\frac{\gamma-1}{2}}+h^{\frac{1-\gamma}{2}})\Big(h^{s+2}(1+h^{\frac{\gamma-1}{2}})\Vert \mathbf{u}\Vert_{s+2}\\
&+h^{n+2}(1+h^{\frac{\gamma-1}{2}})\Vert p\Vert_{n+1}+h^{k+1}(1+h^{\frac{1-\gamma}{2}}+h^{1-\gamma})\Vert \mathbf{u}\Vert_{k+1}\Big).
\end{aligned}
\end{equation*}
Moreover, the optimal order of convergence is achieved when $\gamma=1$.
\end{theorem}
\begin{proof}
A proof can be given by following the same line as that of Theorem \ref{err_esti_L2}, but in a simpler manner. For example, we have in \eqref{eq:proj_dual_final} that $l_5(\mathbf{u}, Q_h^p \xi)=0$ and $l_5(\varphi, e_h^p)=0$ due to the selection of $n\leq j$.
Next, the following refined estimate holds true
 \begin{equation}\label{eq:L2dual_est_3_1}
\begin{aligned}
l_2(p, Q_h \varphi)&=\sum_{T\in\mathcal{T}_h} \langle (I-Q_h^p)p, (Q_b-Q_0)\varphi\cdot\mathbf{n}\rangle_{\partial T}\\
&=\sum_{T\in\mathcal{T}_h} \langle (I-Q_h^p)p, (I-Q_0)\varphi\cdot\mathbf{n}\rangle_{\partial T}\\
&\leq Ch^{n+1}\Vert p\Vert_{n+1}h\Vert\varphi\Vert_2,
\end{aligned}
\end{equation}
where we have used the fact that $\sum_{T\in\mathcal{T}_h} \langle (I-Q_h^p)p, (Q_b-I)\varphi\cdot\mathbf{n}\rangle_{\partial T}=0$ if $n\leq j$. The rest of the proof remains no change, and is thus ommited.
\end{proof}

\medskip
The following is a special case of the error estimate for the Stokes problem with finite elements consisting of the lowest order of polynomials for some of the components.
\begin{corollary}\label{thm:k0}
Assume $k=0$, $l=0$, $n\leq\{m,k+1,j\}$, $\mathbf{u}\in [H^{2}(\Omega)]^d$, and $p\in H^{n+1}(\Omega)$.
Then, there exists a constant $C$ such that
\begin{equation*}
\begin{aligned}
\interleave Q_h\mathbf{u}-\mathbf{u}_h \interleave &\leq C\Big(h (1+h^{\frac{\gamma-1}{2}})\Vert \mathbf{u}\Vert_{2}+h^{n+1}
(1+h^{\frac{\gamma-1}{2}})\Vert p\Vert_{n+1}+h^{\frac{1-\gamma}{2}}\Vert \mathbf{u}\Vert_{1}\Big),\\
\Vert Q_h^p p-p_h \Vert \leq& C\Big(h (1+h^{\frac{\gamma-1}{2}})\Vert \mathbf{u}\Vert_{2}+h^{n+1}(1+h^{\frac{\gamma-1}{2}})\Vert p\Vert_{n+1}\\
&+h^{\frac{1-\gamma}{2}}\Vert \mathbf{u}\Vert_{1}\Big)(1+h^{\frac{1-\gamma}{2}}),
\end{aligned}
\end{equation*}
and
\begin{equation*}
\begin{aligned}
\Vert Q_0\mathbf{u}-\mathbf{u}_0 \Vert  &\leq C(h(1+h^{\frac{\gamma-1}{2}})+h^{\frac{1-\gamma}{2}})\Big( h (1+h^{\frac{\gamma-1}{2}})\Vert \mathbf{u}\Vert_{2}+h^{n+1}(1+h^{\frac{\gamma-1}{2}})\Vert p\Vert_{n+1}\\
&+h^{\frac{1-\gamma}{2}}\Vert \mathbf{u}\Vert_{1}\Big)+C h\Vert \mathbf{u}\Vert_2+C h^{n+1}\Vert p\Vert_{n+1}.
\end{aligned}
\end{equation*}
The best order of convergence is theoretically achieved by choosing $\gamma=0$.
\end{corollary}

\section{Numerical Experiments}\label{sec:6}
In this section we report some numerical results for the weak Galerkin finite element scheme \eqref{sch:WG_FEM} in 2d. For simplicity, we denote by $([P_k]^2,[P_j]^2,[P_l]^{2\times2}, P_m,P_n)$ the element in which the polynomial space $[P_l]^{2\times2}$ is used in the computation of $\delta_w\mathbf{v}$ in the generalized weak gradient \eqref{weakgradient}. Our implementation of \eqref{sch:WG_FEM} is based on the following configurations: $\gamma=1$, $\beta=-1$,  $\mu=0$ when $n\leq j$, and $\mu=1$ when $n>j$.

To demonstrate the performance of the weak Galerkin finite element scheme \eqref{sch:WG_FEM} with various combinations of polynomial spaces, we have implemented the numerical scheme with the following selection of the finite elements: (1) the element $([P_1]^2,[P_0]^2,[P_1]^{2\times2},P_0,P_0)$ with $\mu=0$, (2) the element $([P_2]^2,[P_1]^2,[P_1]^{2\times2},P_0,P_0)$ with $\mu=0$, (3) the WG-variation of the Taylor-Hood element $([P_2]^2,[P_1]^2,[P_1]^{2\times2},P_1,P_1)$ with $\mu=0$, (4) the element $([P_2]^2,[P_1]^2,[P_0]^{2\times2},P_1,P_1)$ with $\mu=0$, and (5) the element $([P_2]^2,[P_1]^2,[P_0]^{2\times2},P_2,P_2)$ with $\mu=1$.

\subsection{Test Case 1: uniform partition}\label{ex:01} This test case assumes an exact solution of $\mathbf{u}=[x^2 y; -x y^2]$ and $p=10(2x-1)(2y-1)$ in the domain $\Omega=(0,1)^2$. {\em Uniform triangular partitions} are employed in the implementation of the numerical scheme \eqref{sch:WG_FEM}. The Dirichlet boundary value and the force term $\mathbf{f}$ are determined by the exact solution.

Table \ref{TT01} shows the numerical results for the element $([P_1]^2,[P_0]^2,[P_1]^{2\times2},P_0,P_0)$. It can be seen that the numerical velocity converges at the rate of $O(h^2)$ in $L^2$ norm and $O(h)$ in an $H^1$-equivalent norm denoted by $\3bar\cdot\3bar$. The pressure approximation exhibits a convergence at the rate of $O(h)$ in $L^2$ norm.  Table \ref{TT01.1} illustrates the numerical performance for the elements $([P_2]^2,[P_1]^2,[P_1]^{2\times2},P_0,P_0)$. It appears that the rate of convergence remains the same as the element $([P_1]^2,[P_0]^2,[P_1]^{2\times2},P_0,P_0)$ even though quadratic and linear functions are employed for the velocity in the interior and on the boundary of each element. This comparison reveals that the pressure approximation plays a significant role in the overall performance of the numerical scheme.

Table \ref{TT02} illustrates the performance of the numerical scheme with the WG-variation of the Taylor-Hood element $([P_2]^2,[P_1]^2,[P_1]^{2\times2},P_1,P_1)$. It can be seen that the numerical velocity converges at the rate of $O(h^2)$ and $O(h^3)$ in $H^1$ and $L^2$ norms, respectively. The convergence for the pressure approximation is seen to be of $O(h^2)$ in $L^2$.

Table \ref{TT03} shows the results for the element $([P_2]^2,[P_1]^2,[P_0]^{2\times2},P_1,P_1)$. Due to the use of piecewise constants for the weak gradient approximation, the convergence of the velocity and the pressure approximations are similar to those in Tables \ref{TT01} and \ref{TT01.1}. This experiment indicates that the approximation of the gradient also plays a significant role for the overall performance of the numerical scheme.

Table \ref{TT04} exhibits the performance of the element $([P_2]^2,[P_1]^2,[P_0]^{2\times2},P_2,P_2)$. Note that as the pressure is approximated by quadratic polynomials which is richer than the velocity space, a stabilizer for the pressure is needed in order to have the desired stability and convergence.
The results in Table \ref{TT04} indicate that the pressure has a convergence of $O(h^{1.6})$ in $L^2$. It is also observed that the velocity has a convergence of $O(h)$ and $O(h^2)$ in $H^1$ and $L^2$ norms, respectively.

All the numerical results are in consistency with the theory developed in previous sections.

\begin{table}[ht]
  \centering \renewcommand{\arraystretch}{1.05}
 \small
 \caption{Test Case \ref{ex:01}: convergence performance of the element $([P_1]^2,[P_0]^2,[P_1]^{2\times2},P_0,P_0)$, $\gamma=1$, and $\mu=0$ (i.e.; no stabilizer $s_2$) }
 \label{TT01}
  \begin{tabular}{ccccccc}
  \hline\noalign{\smallskip}
  \multicolumn{1}{c}{h}&\multicolumn{2}{c}{$\3bar e_h^\mathbf{u}\3bar $}
    &\multicolumn{2}{c}{$\|e^\mathbf{u}_h\|$}
      &\multicolumn{2}{c}{$\|e_h^p\|$}\\
\cline{2-3}\cline{4-5}\cline{6-7}\noalign{\smallskip}
 &error  &order &error  &order &error  &order\\
  \hline
 1/16   &3.2027e-01 &~~~ &9.7584e-03    &~~~    &6.3145e-01   &~~~    \\
 1/32   &1.6542e-01 &0.95 &2.6134e-03    &1.90   &3.0027e-01   &1.07   \\
 1/64   &8.3633e-02 &0.98 &6.6883e-04    &1.97   &1.4387e-01   &1.06  \\
 1/128 &4.1966e-02 &0.99 &1.6849e-04    &1.99   &6.9840e-02   &1.04  \\
\hline
  \end{tabular}
\end{table}

\begin{table}[!th]
  \centering
\caption{Test Case \ref{ex:01}: convergence performance of the element $([P_2]^2,[P_1]^2,[P_1]^{2\times2},P_0,P_0)$, $\gamma=1$, and $\mu=0$ (i.e.; no stabilizer $s_2$)}
 \label{TT01.1}
  \small
  \begin{tabular}{ccccccc}
  \hline\noalign{\smallskip}
  \multicolumn{1}{c}{h}&\multicolumn{2}{c}{$\3bar e_h^\mathbf{u}\3bar $}
    &\multicolumn{2}{c}{$||e^\mathbf{u}_h||$}
      &\multicolumn{2}{c}{$||e_h^p||$}\\
\cline{2-3}\cline{4-5}\cline{6-7}\noalign{\smallskip}
 &error  &order &error  &order &error  &order\\
  \hline
 1/16   &2.3761e-01 &~~~ &4.6357e-03    &~~~    &1.1047e+00   &~~~    \\
 1/32   &1.1980e-01 &0.99 &1.1788e-03    &1.97   &5.5097e-01   &1.00   \\
 1/64   &6.0066e-02 &1.00 &2.9637e-04    &1.99   &2.7437e-01   &1.01  \\
 1/128 &3.0059e-02 &1.00 &7.4228e-05    &2.00   &1.3671e-01   &1.01  \\
\hline
  \end{tabular}
\end{table}

\begin{table}[!th]
  \centering
 \caption{Test Case \ref{ex:01}: performance of the elements $([P_2]^2,[P_1]^2,[P_1]^{2\times2},P_1,P_1)$, $\gamma=1$, and $\mu=0$ (i.e.; no stabilizer $s_2$) }
 \label{TT02}
  \small
  \begin{tabular}{ccccccc}
  \hline\noalign{\smallskip}
  \multicolumn{1}{c}{h}&\multicolumn{2}{c}{$\3bar e_h^\mathbf{u}\3bar $}
    &\multicolumn{2}{c}{$||e^\mathbf{u}_h||$}
      &\multicolumn{2}{c}{$||e_h^p||$}\\
\cline{2-3}\cline{4-5}\cline{6-7}\noalign{\smallskip}
 &error  &order &error  &order &error  &order\\
  \hline
 1/16   &1.9652e-02  &~~~  &3.8689e-04    &~~~    &2.8023e-02    &~~~    \\
 1/32   &4.9268e-03 &2.00  &4.8402e-05    &3.00   &7.0058e-03    &2.00   \\
 1/64   &1.2334e-03 &2.00  &6.0531e-06    &3.00   &1.7515e-03    &2.00  \\
 1/128 &3.0855e-04 &2.00  &7.5683e-07    &3.00   &4.3787e-04    &2.00  \\
\hline
  \end{tabular}
\end{table}

\begin{table}[!th]
  \centering
 \caption{Test Case \ref{ex:01}: performance of the element $([P_2]^2,[P_1]^2,[P_0]^{2\times2},P_1,P_1)$, $\gamma=1$, and $\mu=0$ (i.e.;  no stabilizer $s_2$) }
 \label{TT03}
  \small
  \begin{tabular}{ccccccc}
  \hline\noalign{\smallskip}
  \multicolumn{1}{c}{h}&\multicolumn{2}{c}{$\3bar e_h^\mathbf{u}\3bar $}
    &\multicolumn{2}{c}{$||e^\mathbf{u}_h||$}
      &\multicolumn{2}{c}{$||e_h^p||$}\\
\cline{2-3}\cline{4-5}\cline{6-7}\noalign{\smallskip}
 &error  &order &error  &order &error  &order\\
  \hline
 1/16   &5.4395e-02  &~~~  &7.5953e-04   &~~~   &3.5927e-02  &~~~    \\
 1/32   &2.4019e-02 &1.18  &1.1646e-04   &2.71   &1.3231e-02   &1.44   \\
 1/64   &1.1570e-02 &1.05 &2.2052e-05    &2.40  &5.8757e-03   &1.17  \\
 1/128 &5.7282e-03 &1.01 &4.9661e-06    &2.15  &2.8377e-03  &1.05  \\
\hline
  \end{tabular}
\end{table}

\begin{table}[!th]
  \centering
\caption{Test Case \ref{ex:01}: performance of the element $([P_2]^2,[P_1]^2,[P_0]^{2\times2},P_2,P_2)$, $\gamma=1$, $\beta=-1$, and $\mu=1$ (i.e.;  with stabilizer $s_2$) }
 \label{TT04}
  \small
  \begin{tabular}{ccccccc}
  \hline\noalign{\smallskip}
  \multicolumn{1}{c}{h}&\multicolumn{2}{c}{$\3bar e_h^\mathbf{u}\3bar $}
    &\multicolumn{2}{c}{$||e^\mathbf{u}_h||$}
      &\multicolumn{2}{c}{$||e_h^p||$}\\
\cline{2-3}\cline{4-5}\cline{6-7}\noalign{\smallskip}
 &error  &order &error  &order &error  &order\\
  \hline
 1/16   &6.4026e-02  &~~~  &5.9000e-04   &~~~   &1.2409e-02   &~~~    \\
 1/32   &3.2213e-02 &0.99 &1.4390e-04   &2.04   &4.1278e-03   &1.59   \\
 1/64   &1.6153e-02 &1.00 &3.5472e-05    &2.02  &1.3752e-03   &1.59  \\
 1/128 &8.0874e-03 &1.00 &8.8030e-06    &2.01  &4.6322e-04   &1.57  \\
\hline
  \end{tabular}
\end{table}

\subsection{Test Case 2: non-uniform partitions}\label{ex:02} 
This numerical experiment is based on non-uniform triangular partitions for the domain $\Omega=(0,1)^2$; see Fig. \ref{fig01} for an illustration {\color{blue}(Mesh generation in Gmsh \cite{Geuzaine2009})} . The exact solution is chosen as
$\mathbf{u}=[-\cos x \sin y; \sin x \cos y]$ and $p=e^{x^2}\sin y$.
The boundary value and the force term $\mathbf{f}$ are determined by the exact solution.

Tables \ref{TT05} - \ref{TT08} illustrate the numerical results of the generalized weak Galerkin method with five different elements. Tables \ref{TT05}, \ref{TT05.1}, and \ref{TT07} show the numerical performance of the scheme with linear and quadratic convergence for the velocity in the triple-bar norm and the $L^2$ norm and linear convergence for the pressure in
$L^2$. Table \ref{TT06} suggests a convergence of order $O(h^2)$ and $O(h^3)$ for the velocity in the energy (i.e., triple-bar) and $L^2$ norm with the WG-variation of the Taylor-Hood element $([P_2]^2,[P_1]^2,[P_1]^{2\times2},P_1,P_1)$. Note that the pressure has quadratic convergence in $L^2$ as predicted by the theory.

\begin{figure}[!th]
\centering
\subfigure[~]{
\begin{minipage}{5.5cm}
\centering
\includegraphics[width = 4.5cm,height =4cm]{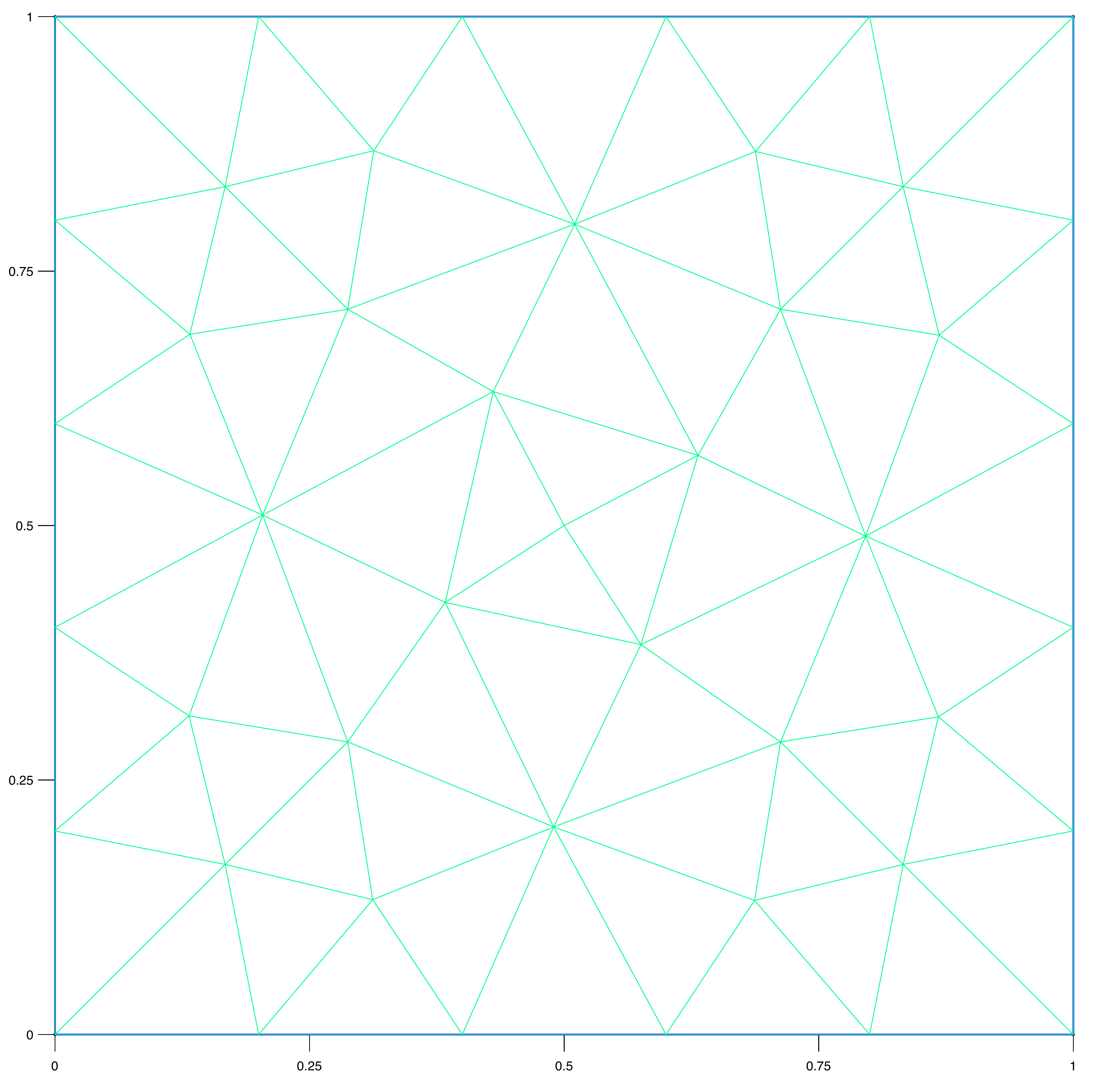}
\end{minipage}
}
\subfigure[~]{
\begin{minipage}{5.5cm}
\centering
\includegraphics[width = 4.5cm,height =4cm]{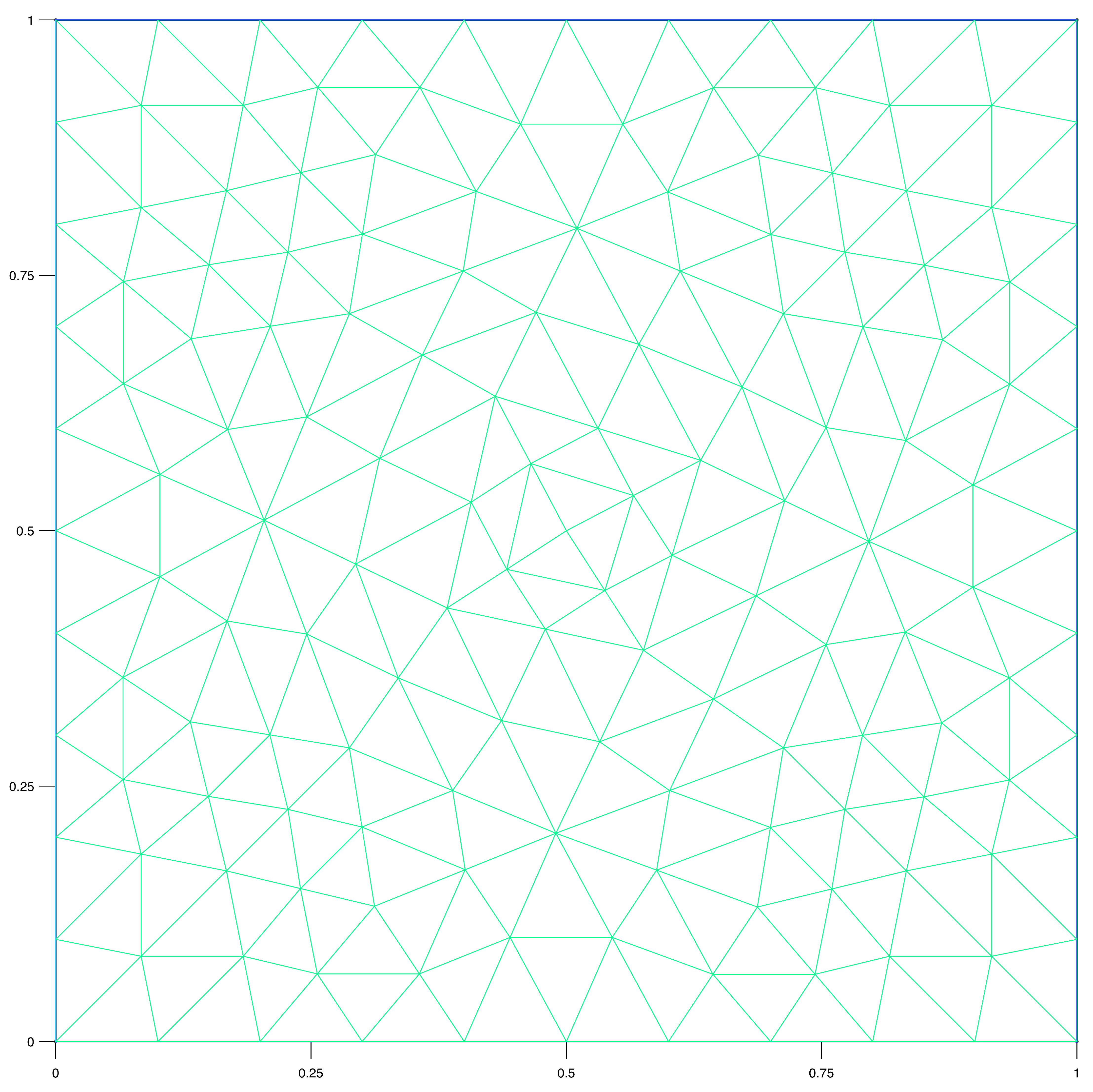}
\end{minipage}
}
\caption{Domain partitions used in Test Case \ref{ex:02}: (a) h=1/5, (b) h=1/10}
\label{fig01}
\end{figure}

\begin{table}[!th]
  \centering
 \caption{ Test Case \ref{ex:02}: performance of the element $([P_1]^2,[P_0]^2,[P_1]^{2\times2},P_0,P_0)$, $\gamma=1$, and $\mu=0$ (i.e.;  no stabilizer $s_2$) }
 \label{TT05}
  \small
  \begin{tabular}{ccccccc}
  \hline\noalign{\smallskip}
  \multicolumn{1}{c}{h}&\multicolumn{2}{c}{$\3bar e_h^\mathbf{u}\3bar $}
    &\multicolumn{2}{c}{$||e^\mathbf{u}_h||$}
      &\multicolumn{2}{c}{$||e_h^p||$}\\
\cline{2-3}\cline{4-5}\cline{6-7}\noalign{\smallskip}
 &error  &order &error  &order &error  &order\\
  \hline
 1/10  &4.6024e-02  &~~~  &1.5262e-03   &~~~   &1.7364e-02   &~~~   \\
 1/20  &2.3941e-02 &0.94 &4.2122e-04   &1.86   &7.4176e-03   &1.23   \\
 1/40  &1.2212e-02 &0.97 &1.1156e-04    &1.92  &2.9885e-03   &1.31  \\
 1/80  &6.1590e-03 &0.99 &2.8710e-05    &1.96  &1.1600e-03  &1.37  \\
\hline
  \end{tabular}
\end{table}

\begin{table}[!th]
  \centering
\caption{Test Case \ref{ex:02}: performance of the element $([P_2]^2,[P_1]^2,[P_1]^{2\times2},P_0,P_0)$, $\gamma=1$, and $\mu=0$ (i.e.; no stabilizer $s_2$)}
 \label{TT05.1}
  \small
  \begin{tabular}{ccccccc}
  \hline\noalign{\smallskip}
  \multicolumn{1}{c}{h}&\multicolumn{2}{c}{$\3bar e_h^\mathbf{u}\3bar $}
    &\multicolumn{2}{c}{$||e^\mathbf{u}_h||$}
      &\multicolumn{2}{c}{$||e_h^p||$}\\
\cline{2-3}\cline{4-5}\cline{6-7}\noalign{\smallskip}
 &error  &order &error  &order &error  &order\\
  \hline
 1/10  &3.7025e-02  &~~~  &9.6305e-04   &~~~   &5.3891e-02   &~~~   \\
 1/20  &1.8746e-02 &0.98 &2.4911e-04   &1.95   &2.7228e-02   &0.98   \\
 1/40  &9.4147e-03 &0.99 &6.3110e-05    &1.98  &1.3695e-02   &0.99  \\
 1/80  &4.7136e-03 &1.00 &1.5848e-05    &1.99  &6.8706e-03  &1.00  \\
\hline
  \end{tabular}
\end{table}

\begin{table}[!th]
  \centering
 \caption{Test Case \ref{ex:02}: performance of the element $([P_2]^2,[P_1]^2,[P_1]^{2\times2},P_1,P_1)$, $\gamma=1$, and $\mu=0$ (i.e.;  no stabilizer $s_2$) }
 \label{TT06}
  \small
  \begin{tabular}{ccccccc}
  \hline\noalign{\smallskip}
  \multicolumn{1}{c}{h}&\multicolumn{2}{c}{$\3bar e_h^\mathbf{u}\3bar $}
    &\multicolumn{2}{c}{$||e^\mathbf{u}_h||$}
      &\multicolumn{2}{c}{$||e_h^p||$}\\
\cline{2-3}\cline{4-5}\cline{6-7}\noalign{\smallskip}
 &error  &order &error  &order &error  &order\\
  \hline
 1/10  &3.4489e-03  &~~~  &1.0439e-04   &~~~   &9.2655e-04  &~~~    \\
 1/20  &8.6811e-04 &1.99 &1.3100e-05   &2.99   &2.3812e-04   &1.96   \\
 1/40  &2.1773e-04 &2.00 &1.6406e-06    &3.00  &6.0441e-05   &1.98\\
 1/80  &5.4513e-05 &2.00 &2.0528e-07    &3.00 &1.5235e-05   &1.99  \\
\hline
  \end{tabular}
\end{table}

\begin{table}[!th]
  \centering
 \caption{Test Case \ref{ex:02}: performance of the element $([P_2]^2,[P_1]^2,[P_0]^{2\times2},P_1,P_1)$, $\gamma=1$, and $\mu=0$ (i.e.; no stabilizer $s_2$) }
 \label{TT07}
  \small
  \begin{tabular}{ccccccc}
  \hline\noalign{\smallskip}
  \multicolumn{1}{c}{h}&\multicolumn{2}{c}{$\3bar e_h^\mathbf{u}\3bar $}
    &\multicolumn{2}{c}{$||e^\mathbf{u}_h||$}
      &\multicolumn{2}{c}{$||e_h^p||$}\\
\cline{2-3}\cline{4-5}\cline{6-7}\noalign{\smallskip}
 &error  &order &error  &order &error  &order\\
  \hline
 1/10  &3.9530e-02  &~~~  &5.9917e-04   &~~~   &2.0136e-02  &~~~    \\
 1/20  &1.9784e-02 &1.00 &1.4376e-04   &2.06   &9.8329e-03   &1.03   \\
 1/40  &9.9032e-03 &1.00 &3.5542e-05    &2.02  &4.8647e-03   &1.02 \\
 1/80  &4.9538e-03 &1.00 &8.8588e-06    &2.00 &2.4230e-03   &1.01  \\
\hline
  \end{tabular}
\end{table}

\begin{table}[!th]
  \centering
 \caption{Test Case \ref{ex:02}: performance of the element $([P_2]^2,[P_1]^2,[P_0]^{2\times2},P_2,P_2)$, $\gamma=1$, $\beta=-1$, and $\mu=1$.}
 \label{TT08}
  \small
  \begin{tabular}{ccccccc}
  \hline\noalign{\smallskip}
  \multicolumn{1}{c}{h}&\multicolumn{2}{c}{$\3bar e_h^\mathbf{u}\3bar $}
    &\multicolumn{2}{c}{$||e^\mathbf{u}_h||$}
      &\multicolumn{2}{c}{$||e_h^p||$}\\
\cline{2-3}\cline{4-5}\cline{6-7}\noalign{\smallskip}
 &error  &order &error  &order &error  &order\\
  \hline
 1/10  &5.4373e-02  &~~~  &1.0275e-03   &~~~  &1.8224e-02  &~~~    \\
 1/20  &2.7574e-02 &0.98 &2.5260e-04   &2.02   &6.1178e-03   &1.57   \\
 1/40  &1.3875e-02 &0.99 &6.2548e-05    &2.01  &2.0661e-03   &1.57 \\
 1/80  &6.9568e-03 &1.00 &1.5551e-05    &2.01  &7.1879e-04   &1.52  \\
\hline
  \end{tabular}
\end{table}

\subsection{Test Case 3: lid-driven cavity}\label{ex:03}
Consider the lid-driven cavity problem with incompressible fluid flow on the unit square $\Omega=(0,1)^2$ described as in \cite{Zienkiewicz2014}. This test problem assumes a slip boundary on the top of the cavity $\Gamma_1=\{(x, y) | y=1\}$ which moves with velocity $\mathbf{u}=[1;0]$. No-slip conditions are assumed on the rest of the boundary $\Gamma_2=\partial \Omega\setminus\Gamma_1$. The pressure constraint is $p=0$ at the point $(0,0)$.

Fig. \ref{fig02} and Fig. \ref{fig03} depicts the numerical approximation of the lid-driven cavity problem when the element $([P_2]^2,[P_1]^2,[P_0]^{2\times2},P_1,P_1)$ is employed with $\gamma=1$ and $\mu=0$.  The plots are based on a uniform triangulation of the domain with meshsize $h=1/16$. Fig. \ref{fig02} shows the velocity vector and the pressure distribution where the presures is scaled by $p_h:=p_h-\int_{\Omega}p_hd\Omega$ .
Fig. \ref{fig03} illustrates the contour plot for the two velocity components.

\begin{figure}[!th]
\centering
\subfigure[~]{
\begin{minipage}{5.5cm}
\centering
\includegraphics[width = 5.0cm, height =4.5cm]{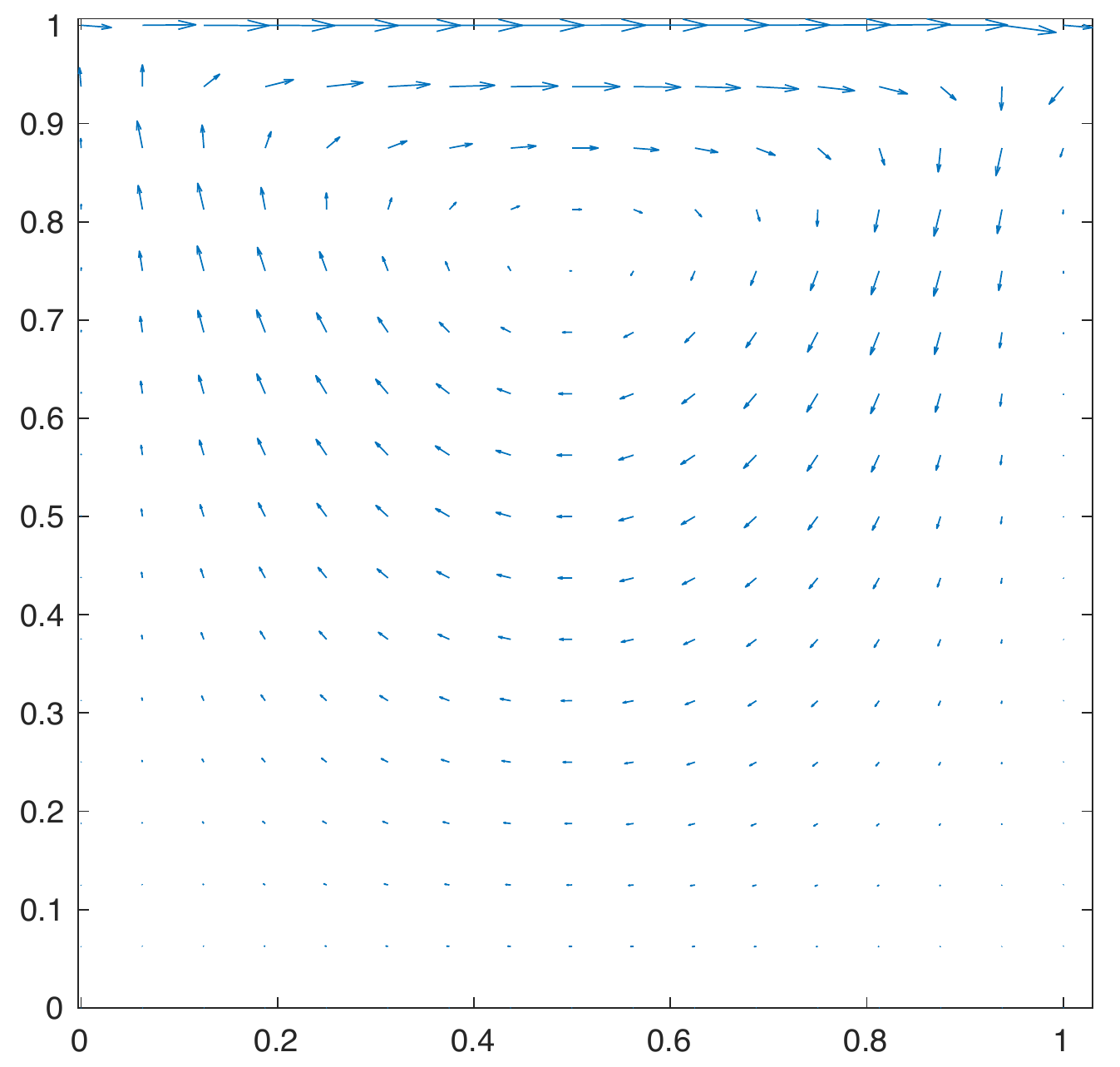}
\end{minipage}
}
\subfigure[~]{
\begin{minipage}{5.5cm}
\centering
\includegraphics[width = 5.0cm, height =4.5cm]{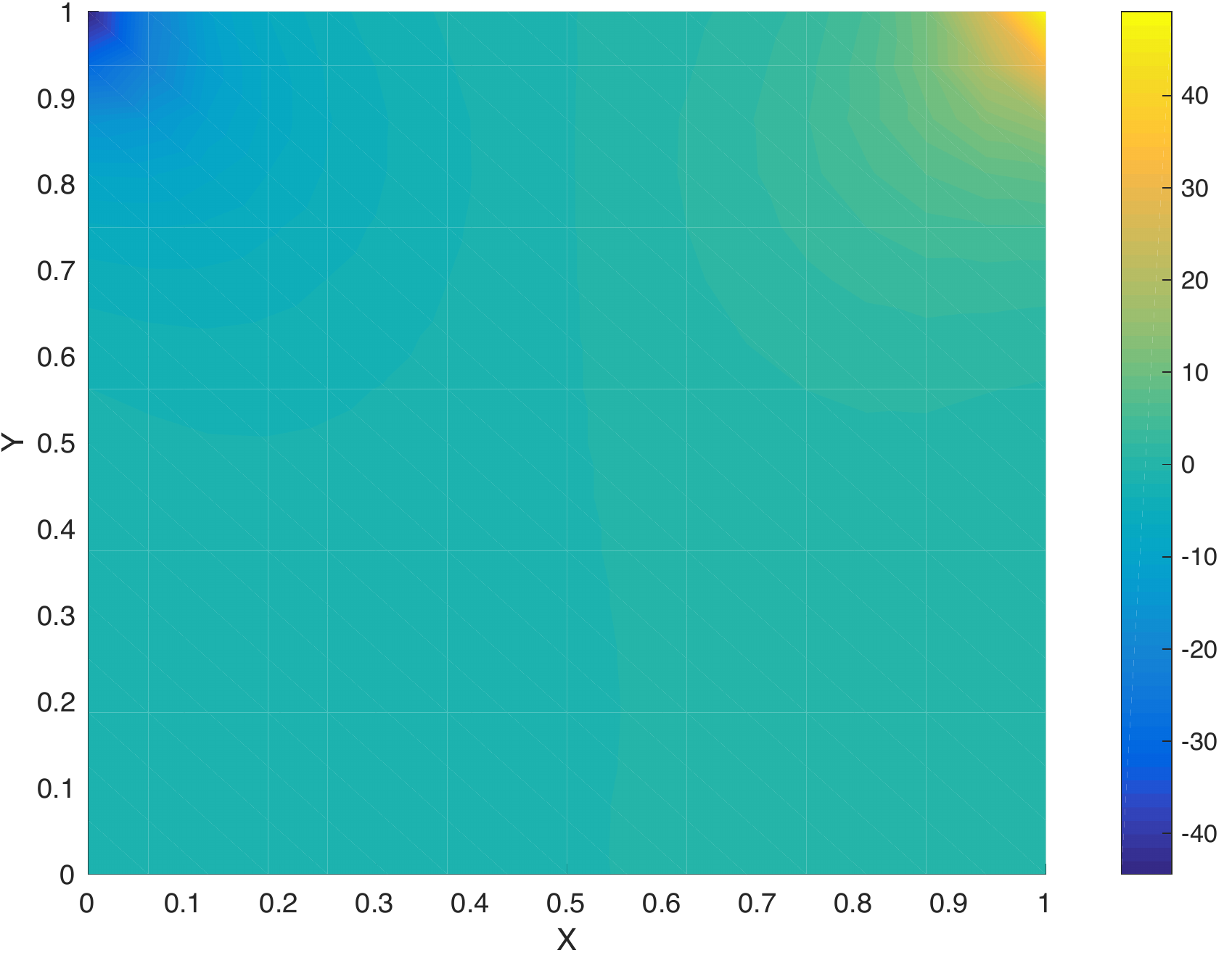}
\end{minipage}
}
\caption{Lid-driven cavity: The numerical solutions with $([P_2]^2,[P_1]^2,[P_0]^{2\times2},P_1,P_1)$: (a) velocity field, (b) pressure.}
\label{fig02}
\end{figure}

\begin{figure}[!th]
\centering
\begin{minipage}{12cm}
\centering
\includegraphics[width = 12cm, height =4.5cm]{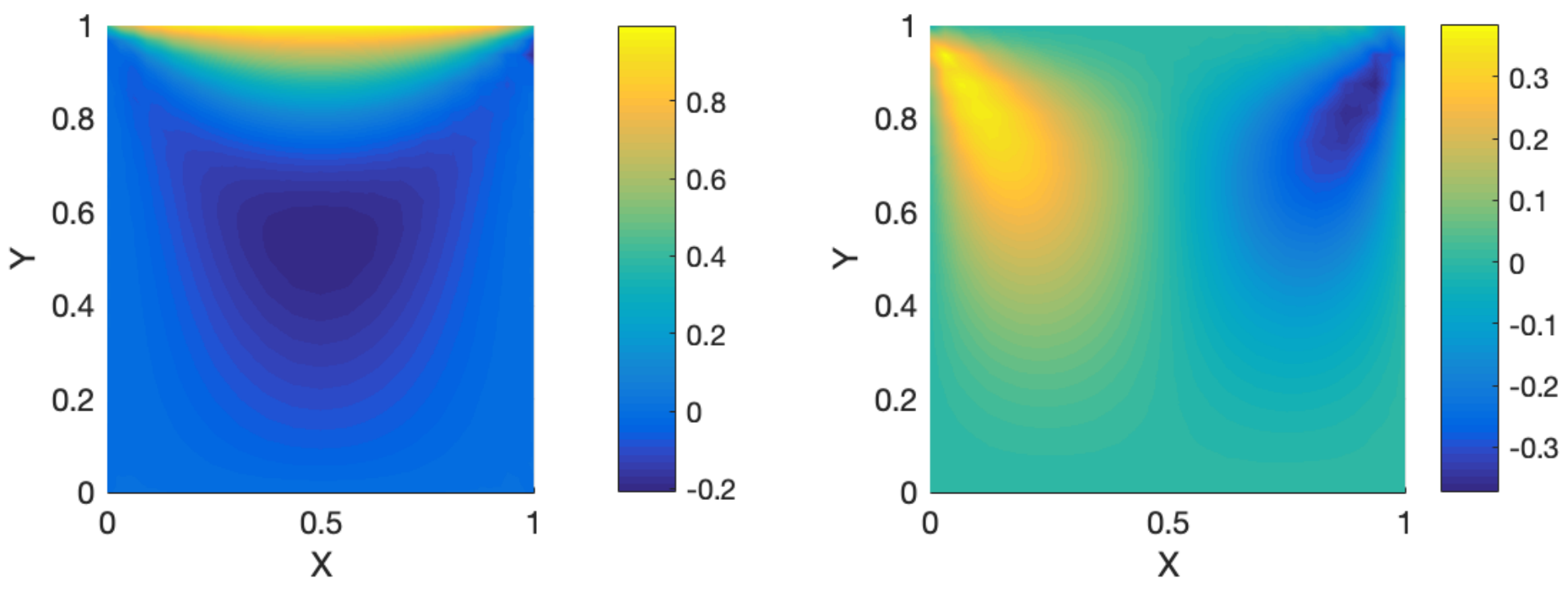}
\end{minipage}
\caption{Lid-driven cavity: Contour plots of the velocity with the element $([P_2]^2,[P_1]^2,[P_0]^{2\times2},P_1,P_1)$:  $ \mathbf{u}_1$(left);  $\mathbf{u}_2$(right).}
\label{fig03}
\end{figure}

This lid-driven cavity problem for incompressible fluid fluid was also approximated by using the element $([P_2]^2,[P_1]^2,[P_1]^{2\times2},P_1,P_1)$ with $\gamma=1$ and $\mu=0$.  The resulting numerical solutions are illustrated in  Fig. \ref{fig04} and Fig. \ref{fig05}. A comparison of the two computations for the lid-driven problem indicates that the WG-variation of the Taylor-Hood element $([P_2]^2,[P_1]^2,[P_1]^{2\times2},P_1,P_1)$ seems to provide numerical solutions with shaper resolution.

\begin{figure}[!th]
\centering
\subfigure[~]{
\begin{minipage}{5.5cm}
\centering
\includegraphics[width = 5.0cm, height =4.5cm]{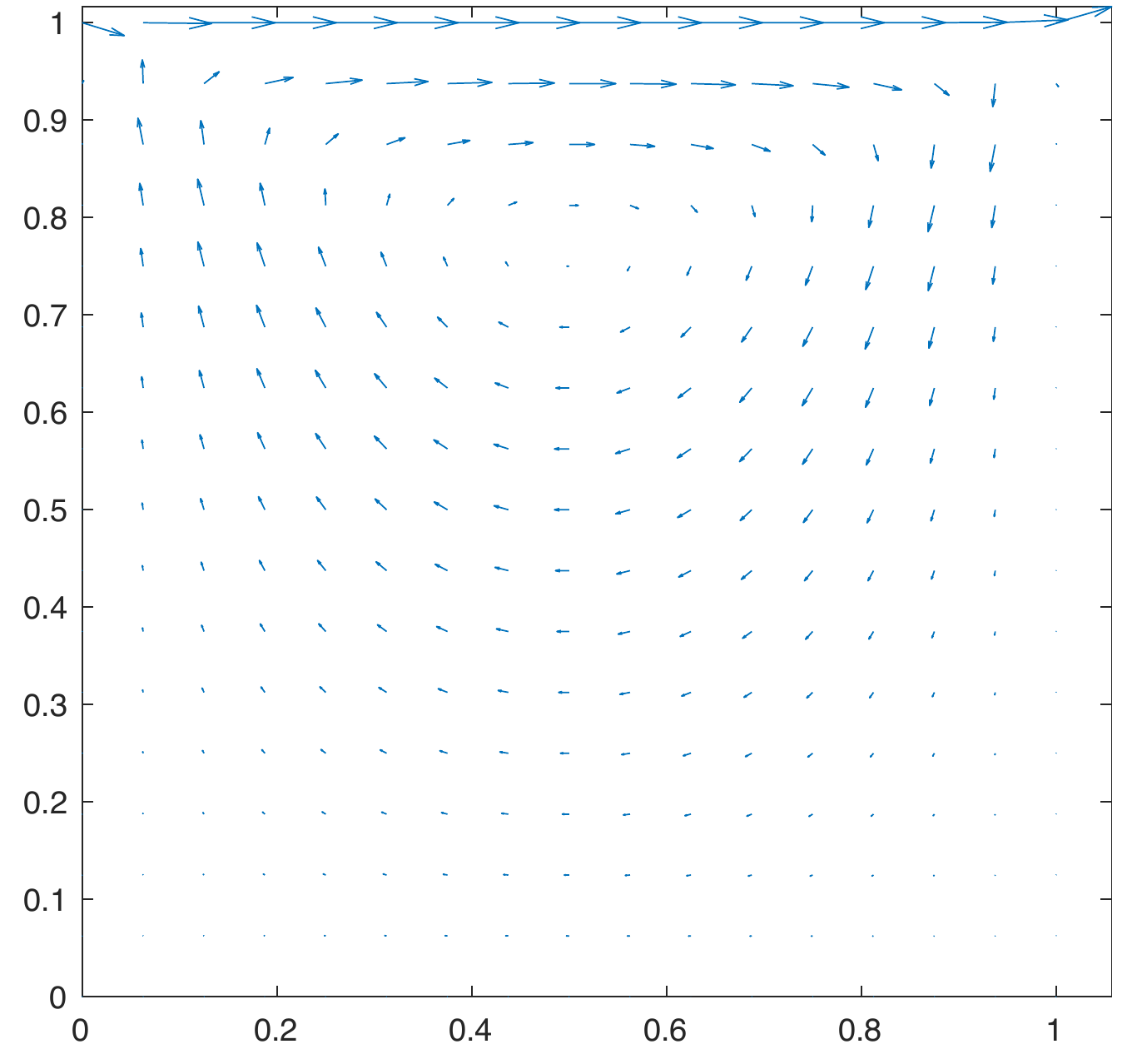}
\end{minipage}
}
\subfigure[~]{
\begin{minipage}{5.5cm}
\centering
\includegraphics[width = 5.0cm, height =4.5cm]{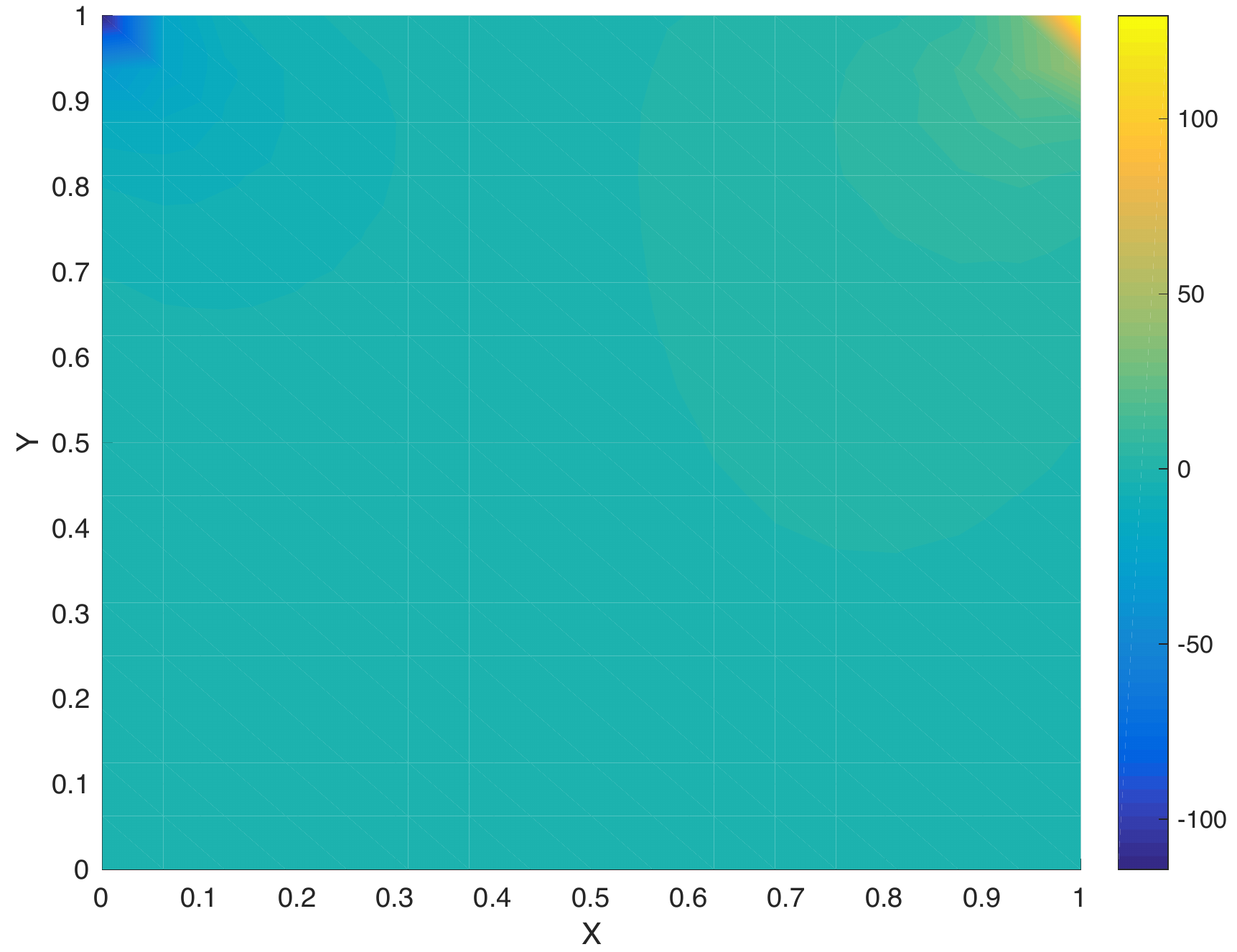}
\end{minipage}
}
\caption{Lid-driven cavity by using the element $([P_2]^2,[P_1]^2,[P_1]^{2\times2},P_1,P_1)$: (a) velocity field, (b) pressure.}
\label{fig04}
\end{figure}

\begin{figure}[!th]
\centering
\begin{minipage}{12cm}
\centering
\includegraphics[width = 12cm, height =4.5cm]{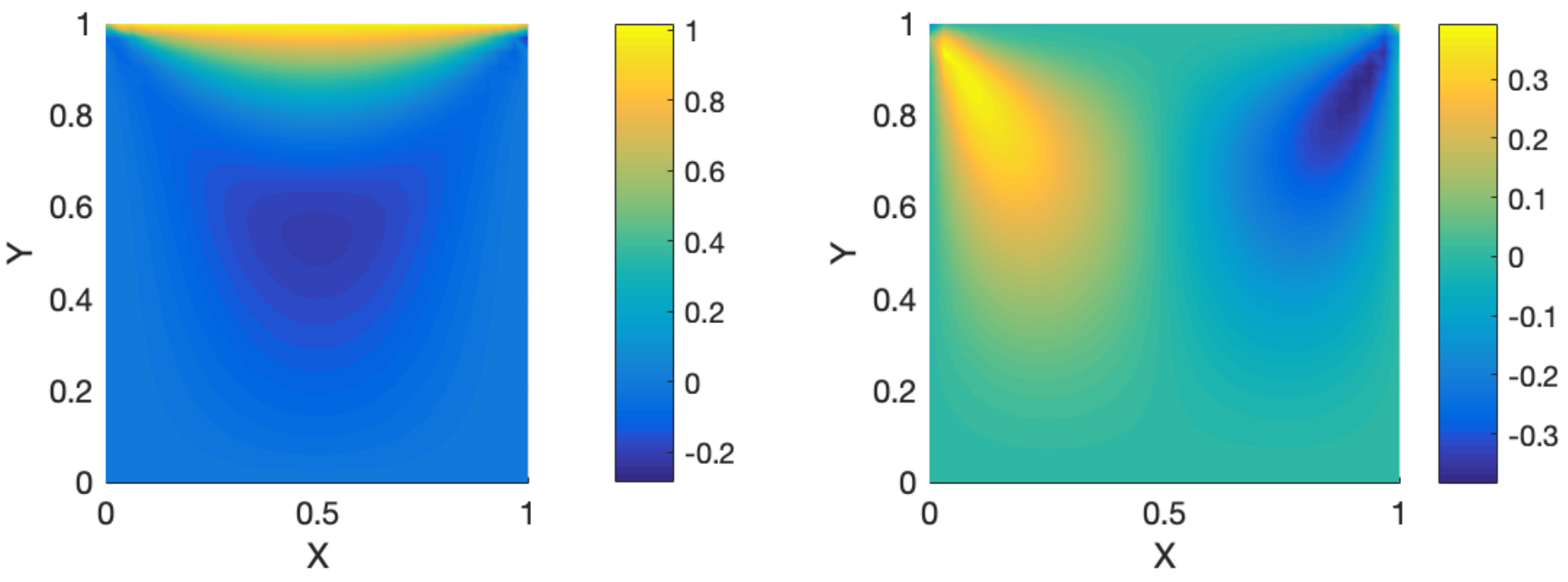}
\end{minipage}
\caption{Lid-driven cavity: Contour plots of the velocity with the element $([P_2]^2,[P_1]^2,[P_1]^{2\times2},P_1,P_1)$:  $ \mathbf{u}_1$(left);  $\mathbf{u}_2$(right).}
\label{fig05}
\end{figure}

\subsection{Test Case 4:  transient flow passing circular objects}\label{ex:04}
We consider the transient flow passing a circular object in two dimensions.
The domain is given as $\Omega=\Omega_1\setminus\Omega_2$, where $\Omega_1=(0, 8)\times(0, 4.5)$ and
$\Omega_2=\{(x, y)| (x-2.25)^2+(y-2.25)^2<0.25^2\}$. No-slip boundary condition is enfored on the inner circle boundary $\Gamma_2=\partial\Omega_2$. The outer boundary condition is $\mathbf{u}=[1;0]$ on $\Gamma_1=\partial\Omega_1$; i.e.,  the velocity in horizontal direction is unit and zero in vertical direction on the slip boundary.

For the problem of transient flow passing a circular object, numerical approximations were obtained by using triangular partitions showed as in Fig. \ref{fig06}. Observe that the elements near the inner circle boundary have smaller size than those near the outer boundary in order to capture the separation movement.
The element $([P_2]^2,[P_1]^2,[P_1]^{2\times2},P_1,P_1)$ was employed in the numerical computation with $\mu=0$.
Fig. \ref{fig07} and Fig. \ref{fig08} show the numerical results on the velocity and pressure.

\begin{figure}[!th]
\centering
\begin{minipage}{8.5cm}
\centering
\includegraphics[width = 8cm, height =4.5cm]{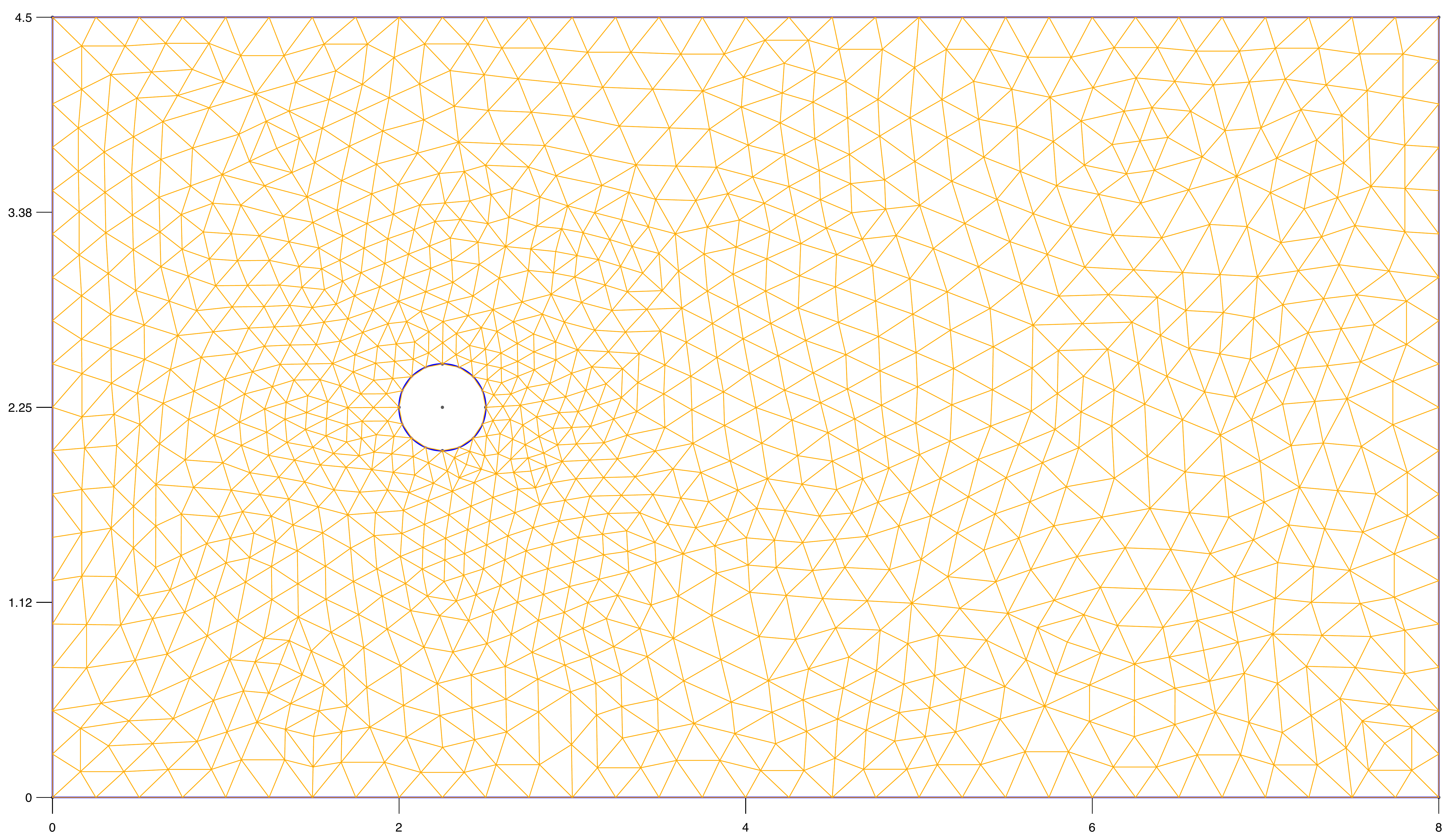}
\end{minipage}
\caption{Illustration of domain partitions for the problem of transient flow.}
\label{fig06}
\end{figure}

\begin{figure}[!th]
\centering
\subfigure[~]{
\begin{minipage}{8.5cm}
\centering
\includegraphics[width = 8cm, height =4.5cm]{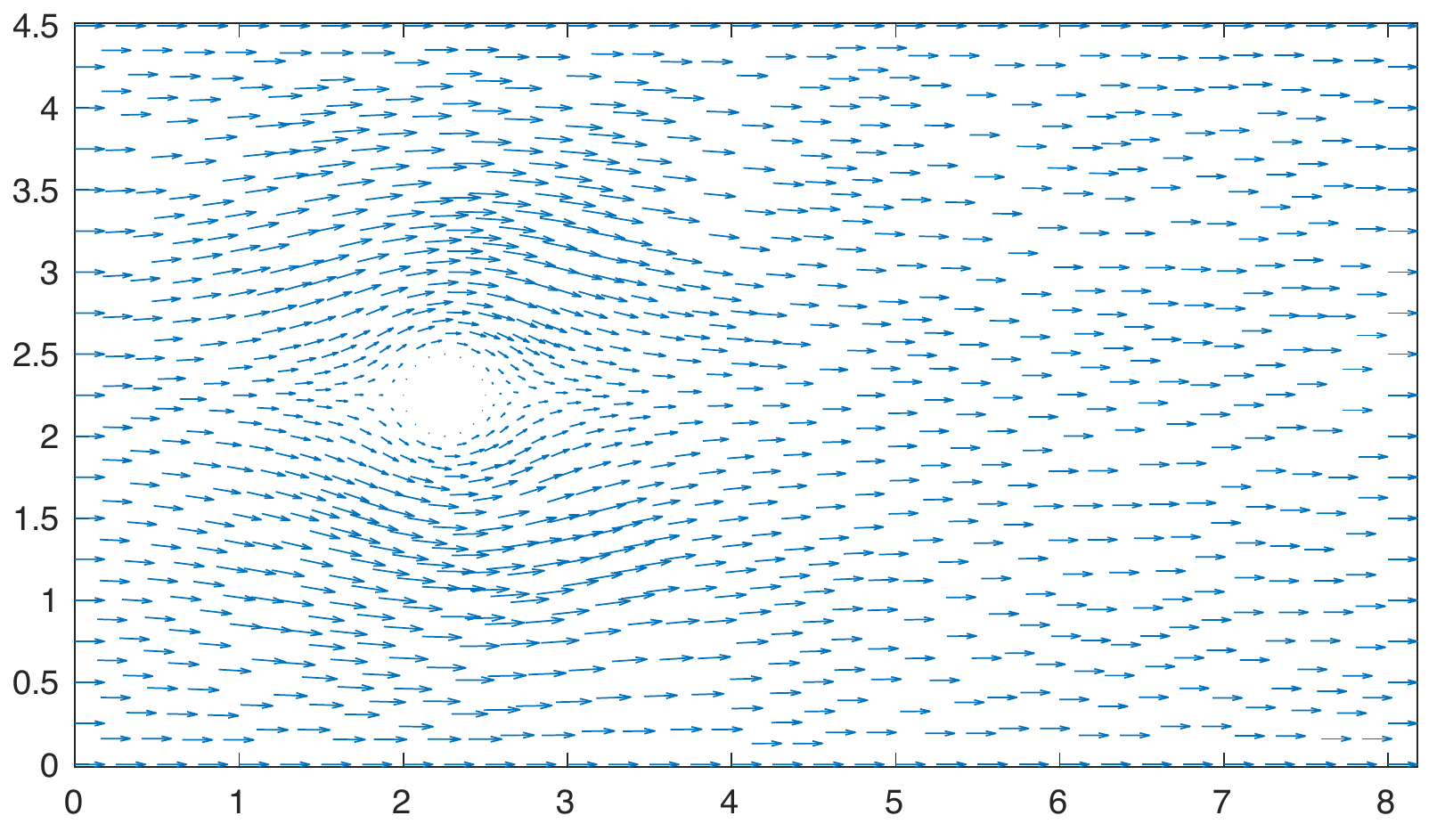}
\end{minipage}
}
\subfigure[~]{
\begin{minipage}{8.5cm}
\centering
\includegraphics[width = 8cm, height =4.5cm]{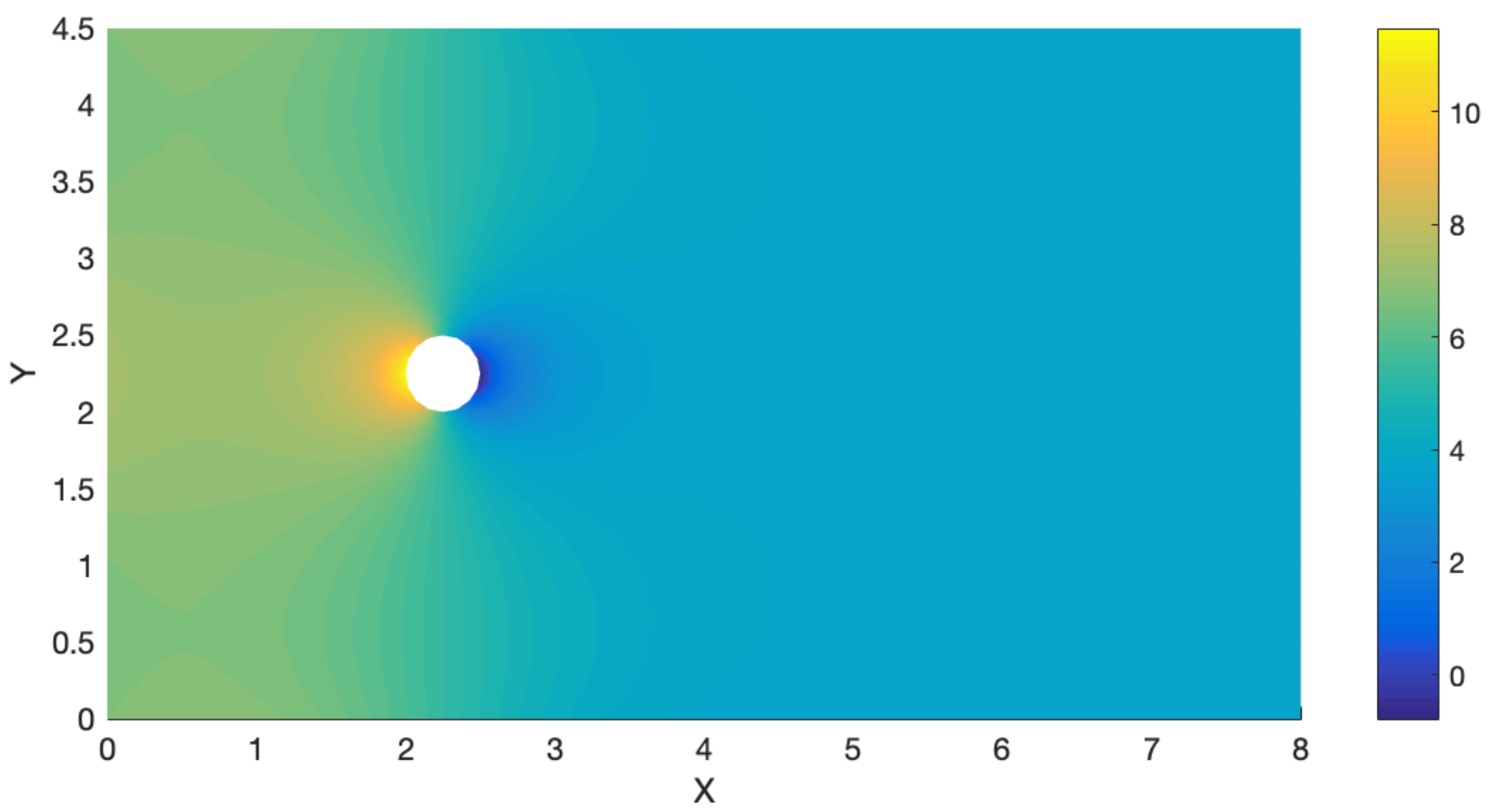}
\end{minipage}
}
\caption{Numerical approximation of transient flow using the element $([P_2]^2,[P_1]^2,[P_1]^{2\times2},P_1,P_1)$: (a) velocity field, (b) pressure profile.}
\label{fig07}
\end{figure}

\begin{figure}[!th]
\centering
\begin{minipage}{8.5cm}
\centering
\includegraphics[width = 8cm, height =9cm]{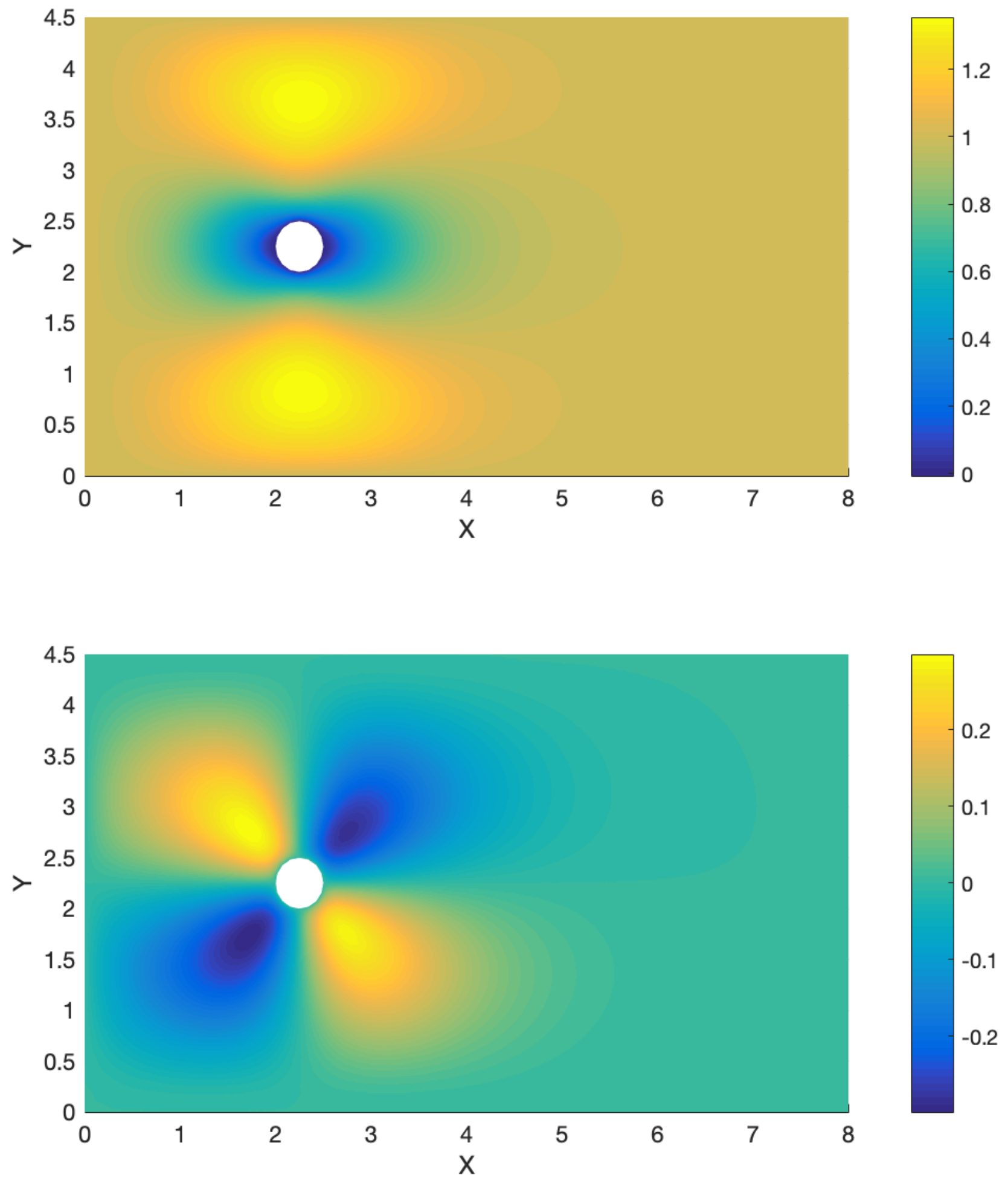}
\end{minipage}
\caption{Transient flow:  the velocity profile from the element $([P_2]^2,[P_1]^2,[P_1]^{2\times2},P_1,P_1)$:  $ \mathbf{u}_1$  (up);  $ \mathbf{u}_2$ (down).}
\label{fig08}
\end{figure}

Our next transient flow problem involves three circular objects in passing. The domain is again set as $\Omega=\Omega_1\setminus\Omega_2$ where $\Omega_1=(0, 8)\times (0, 4.5)$, and the inner domain $\Omega_2$ is composed with three circles $\Omega_2=\Omega_3\cup\Omega_4\cup\Omega_5$:
\begin{equation*}
\begin{aligned}
\Omega_3=\{(x, y)|(x-2.25)^2+(y-2.25)^2&<0.25^2\},\\
\Omega_4=\{(x, y)|(x-4)^2+(y-1.25)^2&<0.25^2\},\\
\Omega_5=\{(x, y)|(x-4)^2+(y-3.25)^2&<0.25^2\}.
\end{aligned}
\end{equation*}
No-slip boundary conditions are imposed on the inner boundary $\Gamma_2=\partial \Omega_2$ and the Dirichlet boundary condition is given on the outer boundary $\Gamma_1=\partial \Omega_1$ as $\mathbf{u}=[1; 0]$.
Numerical approximations are obtained by using the element $([P_2]^2,[P_1]^2,[P_1]^{2\times2},P_1,P_1)$ with the parameter $\mu=0$. The numerical solutions are shown in Fig. \ref{fig010} and Fig. \ref{fig011}.

\begin{figure}[!th]
\centering
\begin{minipage}{8.5cm}
\centering
\includegraphics[width = 8cm, height =4.5cm]{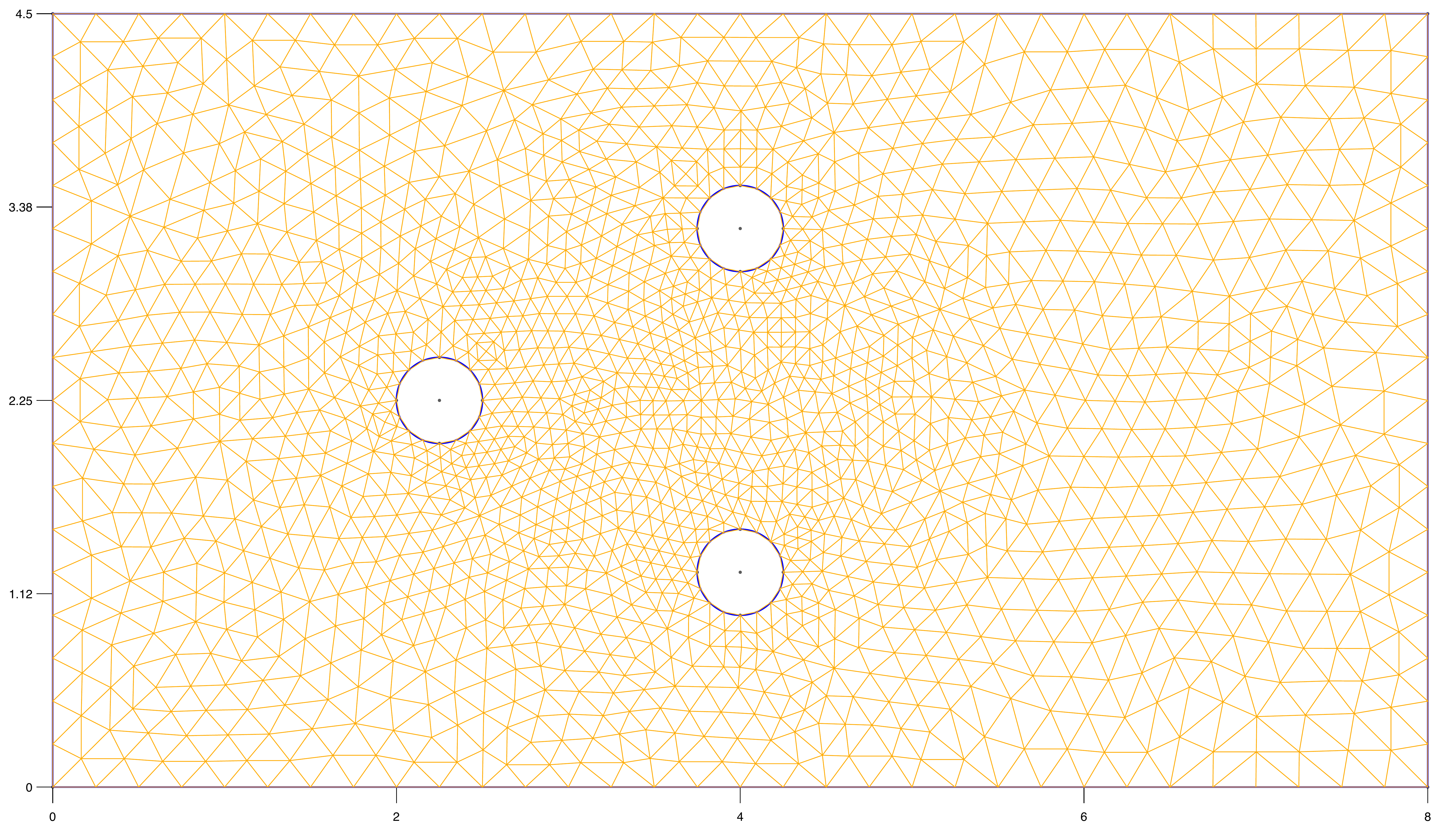}
\end{minipage}
\caption{Illustration of domain partitions for the problem of transient flow passing three circular objects.}
\label{fig09}
\end{figure}

\begin{figure}[!th]
\centering
\subfigure[~]{
\begin{minipage}{8.5cm}
\centering
\includegraphics[width = 8cm, height =4.5cm]{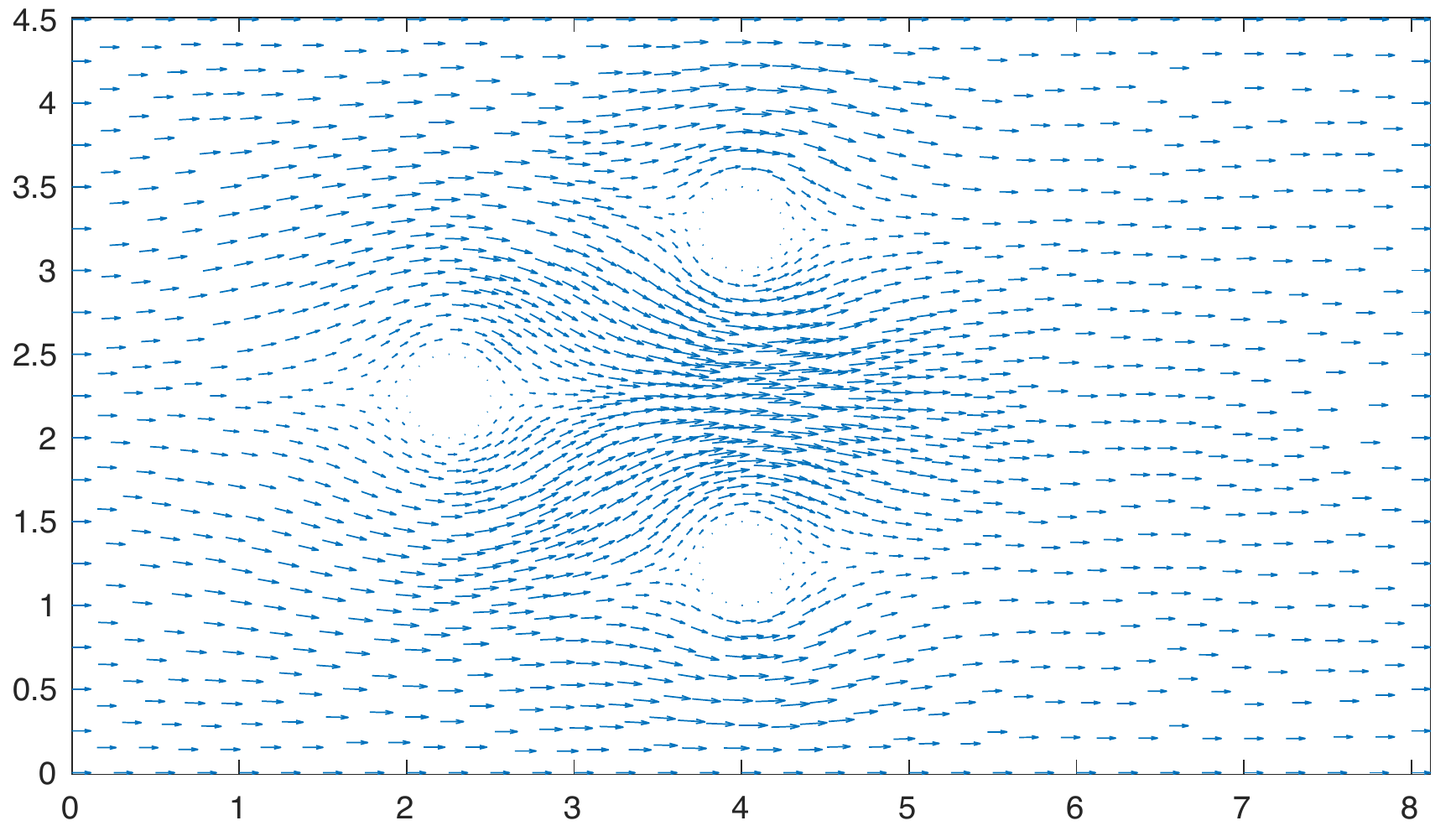}
\end{minipage}
}
\subfigure[~]{
\begin{minipage}{8.5cm}
\centering
\includegraphics[width = 8cm, height =4.5cm]{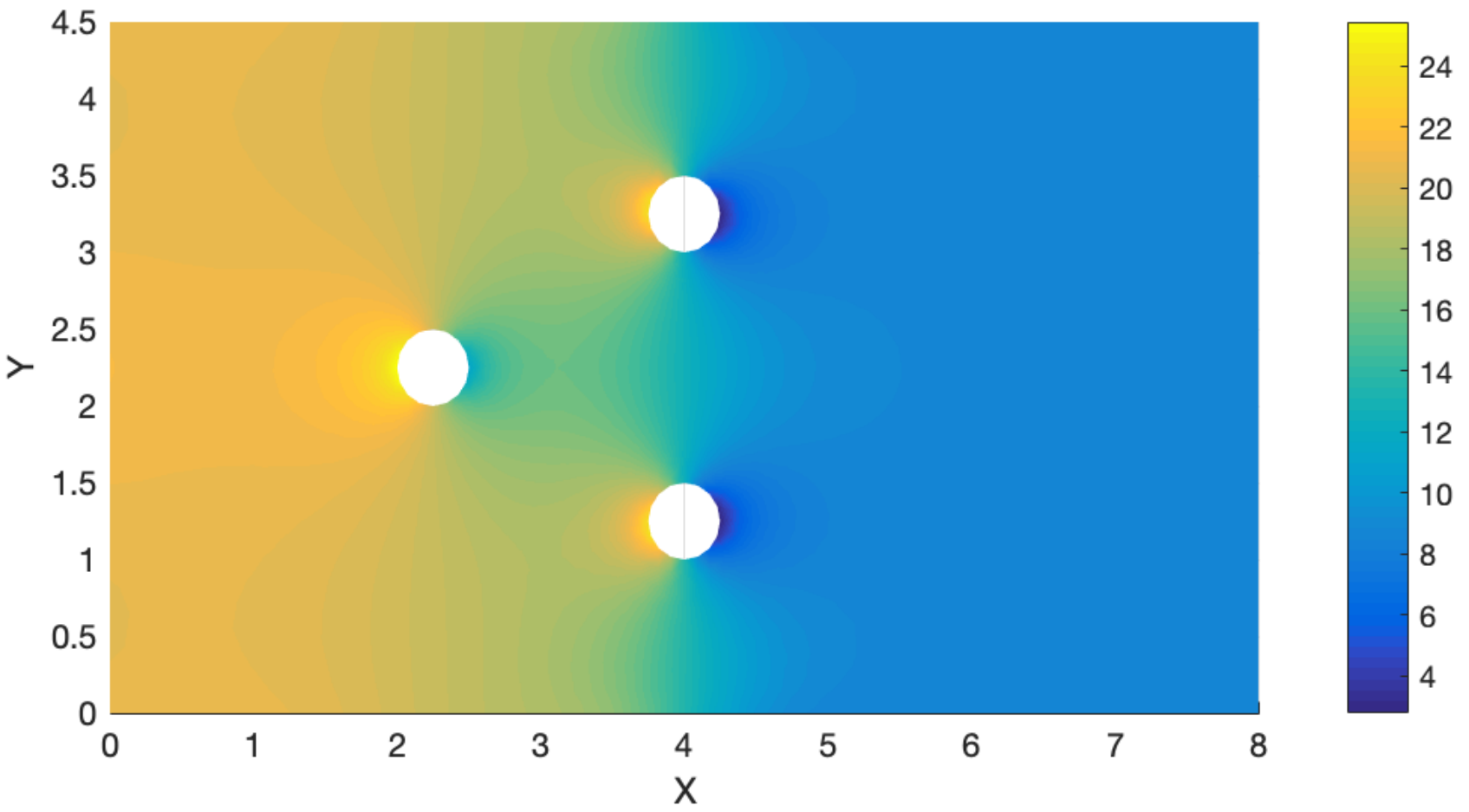}
\end{minipage}
}
\caption{Numerical approximation of a transient flow problem using the element $([P_2]^2,[P_1]^2,[P_1]^{2\times2},P_1,P_1)$: (a) velocity field, (b) pressure contour.}
\label{fig010}
\end{figure}

\begin{figure}[!th]
\centering
\begin{minipage}{8.5cm}
\centering
\includegraphics[width = 8cm, height =9cm]{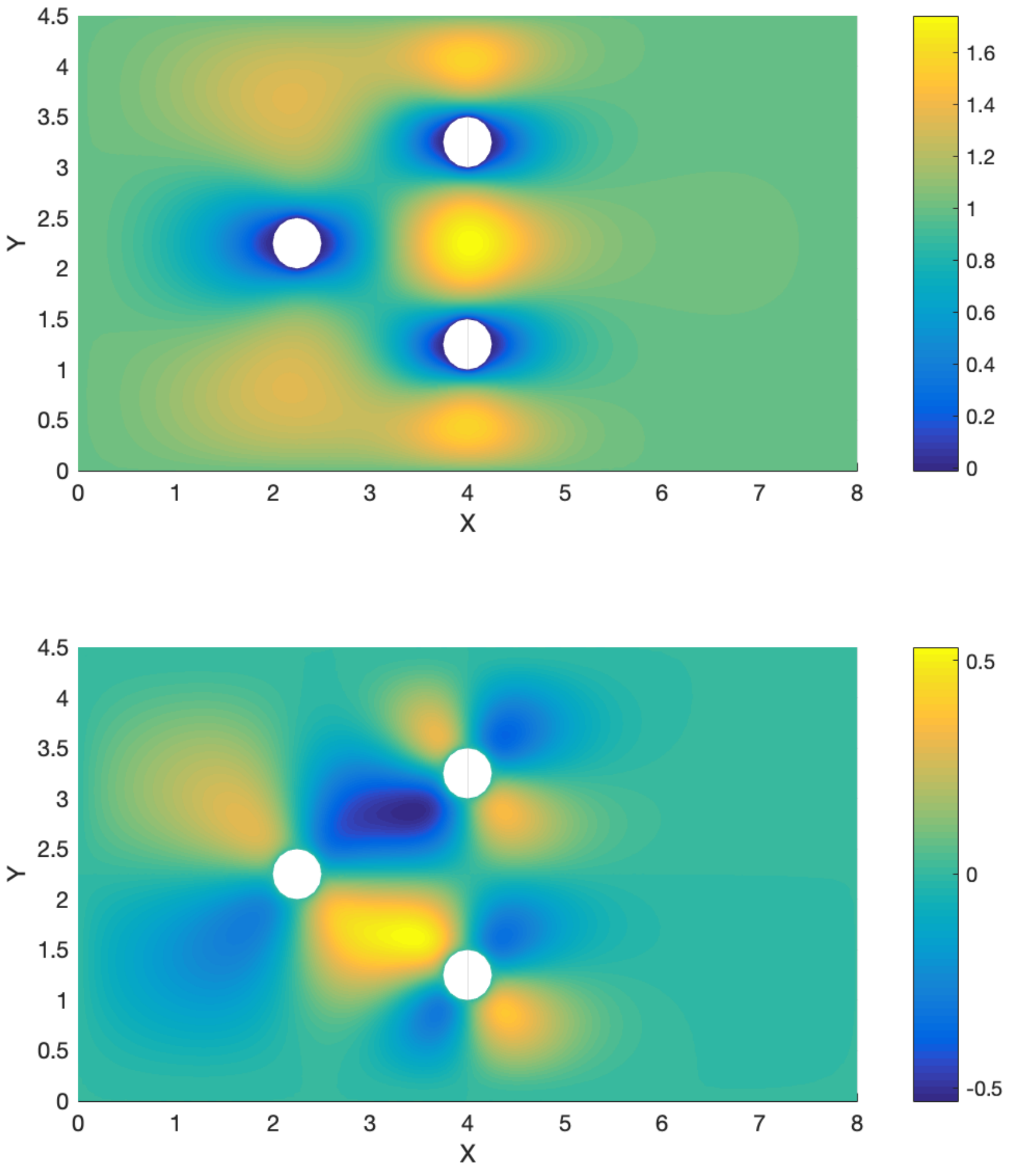}
\end{minipage}
\caption{Transient flow:  the velocity profile from the element $([P_2]^2,[P_1]^2,[P_1]^{2\times2},P_1,P_1)$:  $ \mathbf{u}_1$  (up);  $ \mathbf{u}_2$ (down).}
\label{fig011}
\end{figure}

As our last test case, we consider the transient flow of incompressible fluid passing three circular objects in two dimensions with the following boundary condition: (1) no-slip boundary condition on three circular objects, and (2) Dirichlet boundary condition on the outer boundary which is divided into three segments $\Gamma_1=\Gamma_3\cup\Gamma_4\cup\Gamma_5$:
\begin{eqnarray*}
\Gamma_3&=&\{(x,y)| x=0, 1.5\leq y\leq 2~\mbox{or}~2.5 \leq y\leq 3 \},\\
\Gamma_4&=&\{(x,y)| x=8, 0\leq y\leq 4.5 \},\\
\Gamma_5&=&\partial \Omega_1\setminus\Gamma_3\setminus \Gamma_4.
\end{eqnarray*}
The Dirichlet boundary value is given as $[\alpha; 0]$ on $\Gamma_3$, $[\alpha/4.5; 0]$ on $\Gamma_4$, and $[0; 0]$
on $\Gamma_5$, all with $\alpha=100$. Numerical solutions are obtained by using the element $([P_2]^2,[P_1]^2,[P_1]^{2\times2},P_1,P_1)$ on quasi-regular triangular partitions.  The solutions are depicted in Fig. \ref{fig016} and Fig. \ref{fig017}.

\begin{figure}[!th]
\centering
\subfigure[~]{
\begin{minipage}{8.5cm}
\centering
\includegraphics[width = 8cm, height =4.5cm]{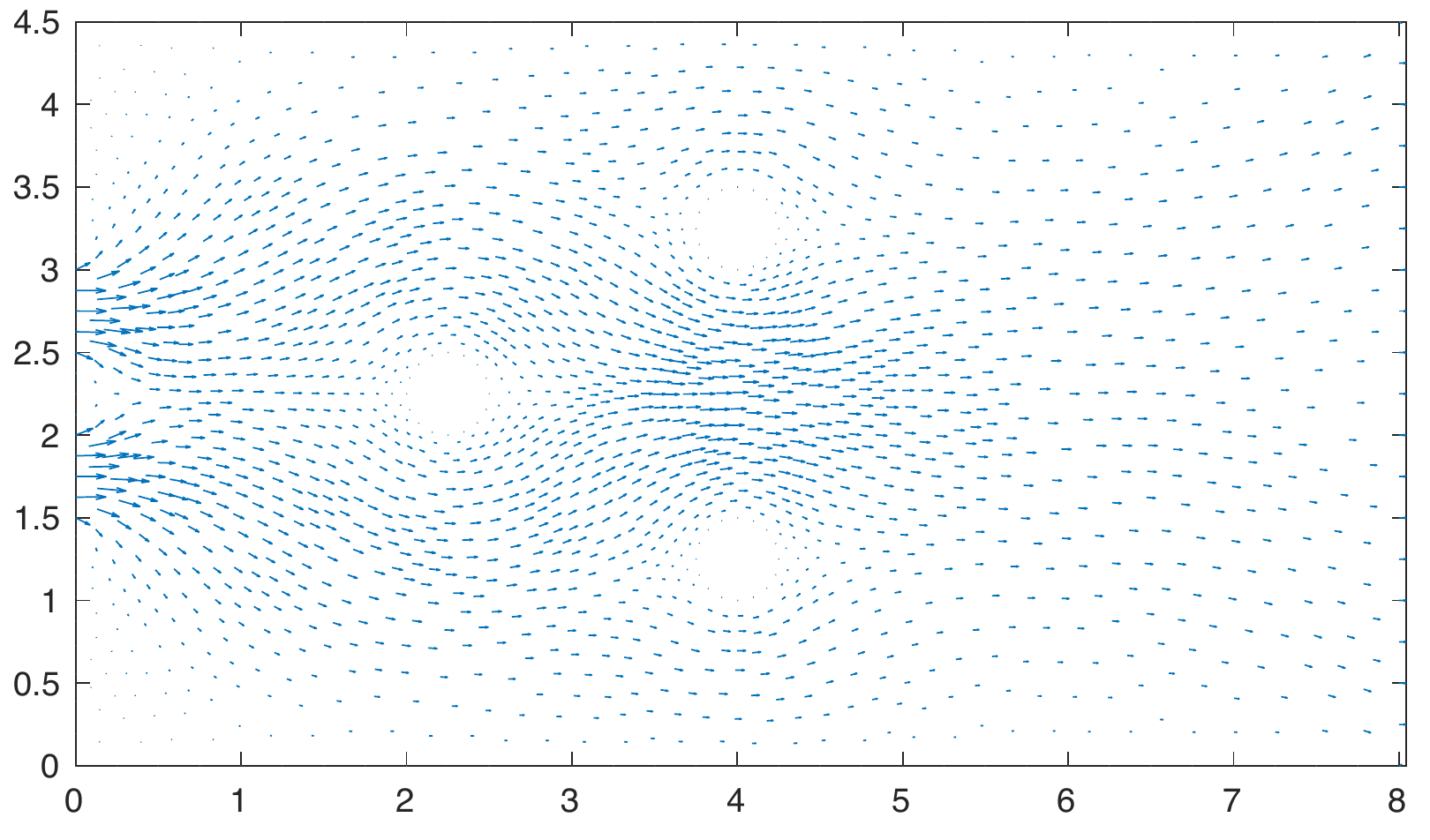}
\end{minipage}
}
\subfigure[~]{
\begin{minipage}{8.5cm}
\centering
\includegraphics[width = 8cm, height =4.5cm]{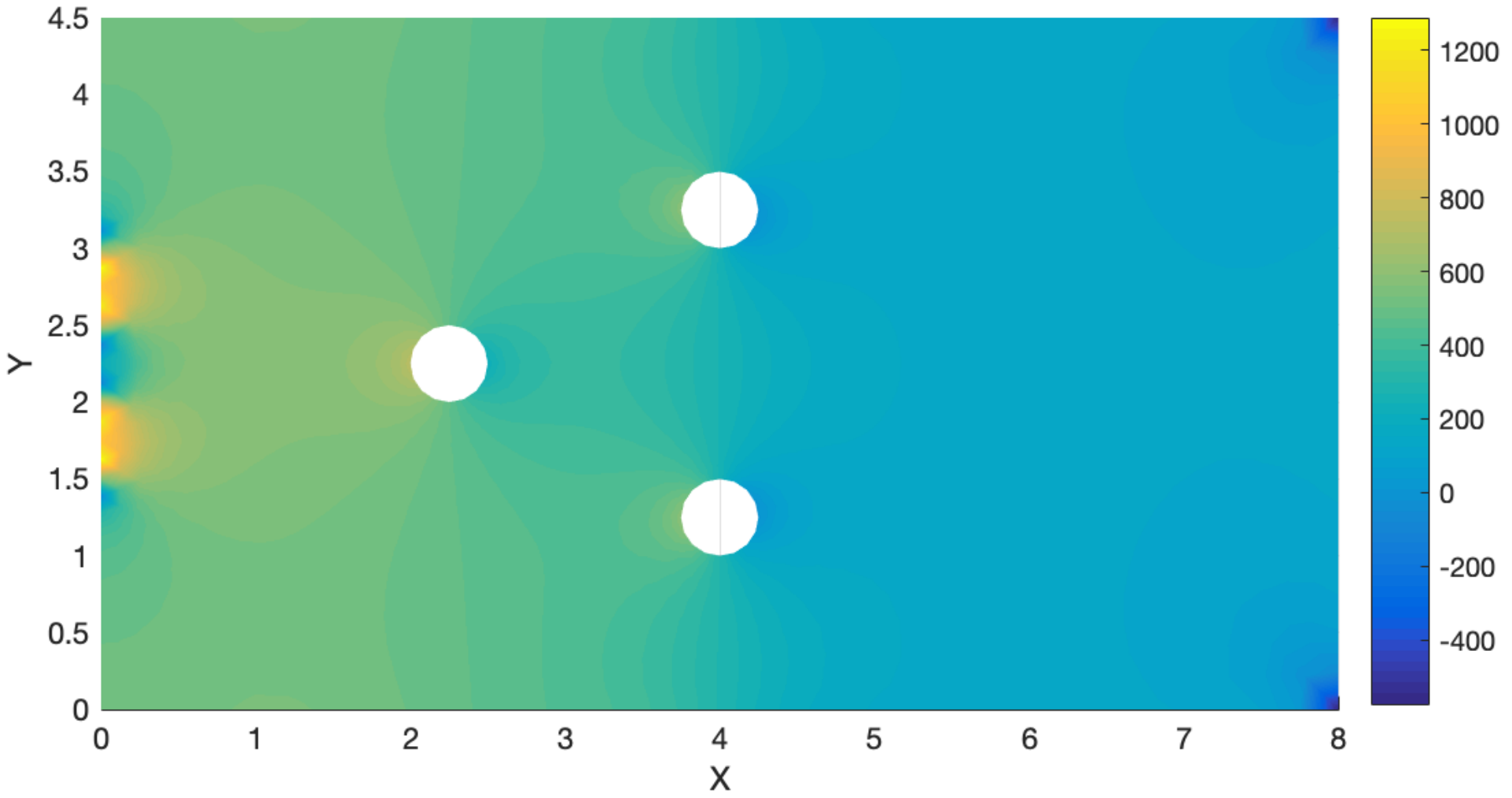}
\end{minipage}
}
\caption{Numerical approximation of a transient flow problem using the element $([P_2]^2,[P_1]^2,[P_1]^{2\times2},P_1,P_1)$: (a) velocity field, (b) pressure contour.}
\label{fig016}
\end{figure}

\begin{figure}[!th]
\centering
\begin{minipage}{8.5cm}
\centering
\includegraphics[width = 8cm, height =9cm]{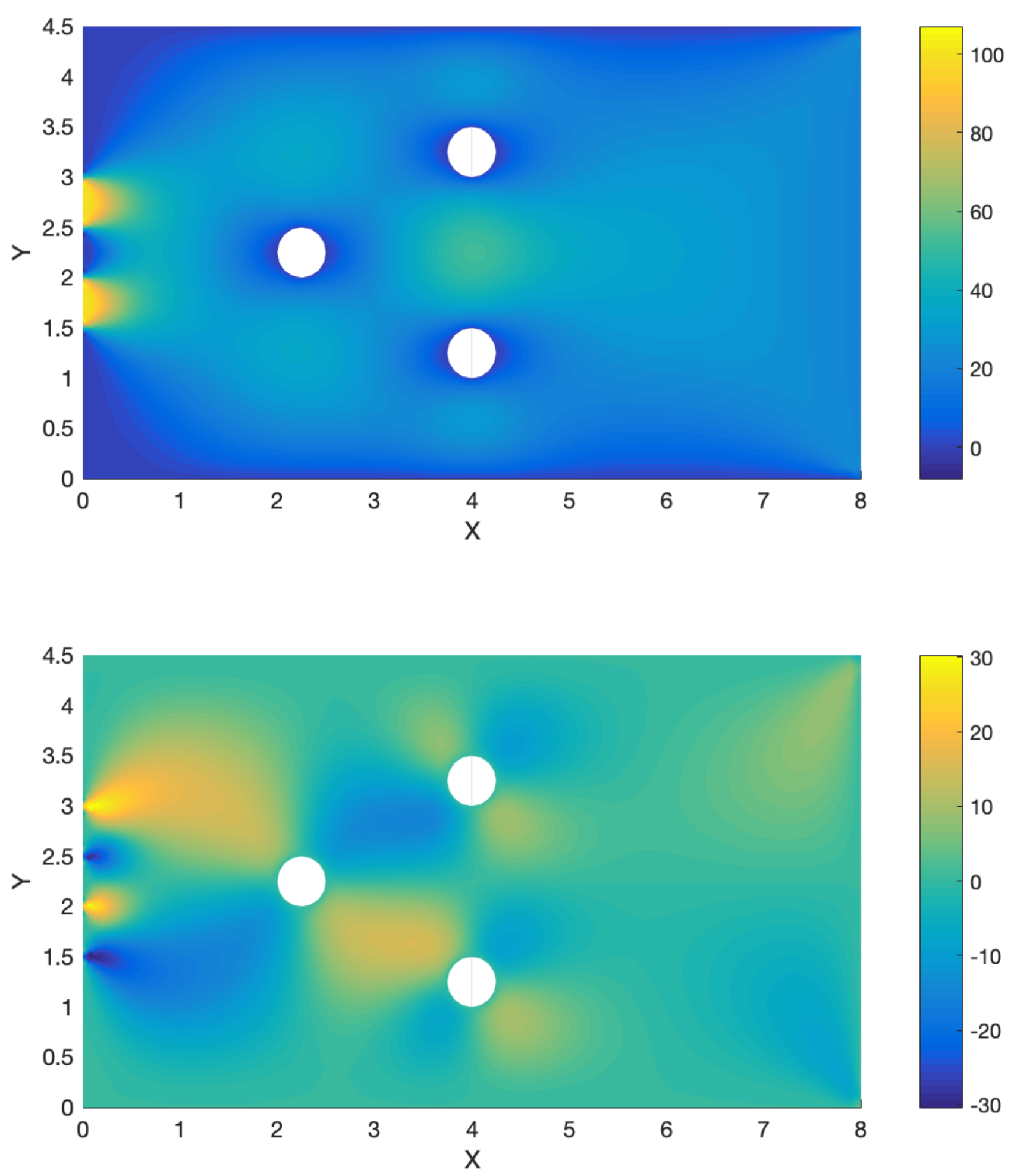}
\end{minipage}
\caption{Transient flow:  the velocity profile from the element $([P_2]^2,[P_1]^2,[P_1]^{2\times2},P_1,P_1)$:  $ \mathbf{u}_1$  (up);  $ \mathbf{u}_2$ (down). }
\label{fig017}
\end{figure}

\section{Conclusions}
In this paper, a generalized weak Galerkin finite element method was developed through the introduction of a generalized weak gradient for the Stokes problem. The new scheme is advantageous in that it allows the use of polynomials of arbitrary order and combination in the construction of the finite elements; yet achieving optimal order of convergence for all the combinations. From a theoretical point of view, the generalized weak Galerkin scheme is based on an extended {\em inf-sup} condition and the use of two stabilizers $s_1$ and $s_2$; one for the velocity and the other for the pressure.

\end{document}